\documentclass[12pt,a4paper]{article}
\usepackage[utf8]{inputenc}
\usepackage[english]{babel}
\usepackage{latexsym,amssymb}
\usepackage{xcolor}
\usepackage{tensor}
\usepackage{authblk}
\author[1]{David A. Henriquez  Bernal} 
\author[1]{Peter  Nejjar} 
\affil[1]{Institute for Mathematics, Potsdam University
 }

\usepackage{cite}
\usepackage[pdftex]{graphicx}
\usepackage{tikz}
\usepackage{amsmath,amsthm,amsfonts,amssymb,bbm}
\usepackage{graphicx,psfrag,subfigure,float,color}
\usepackage{cite}
\usepackage{graphicx}
\usetikzlibrary{arrows}
\usepgflibrary{snakes}
\usetikzlibrary{arrows,snakes,backgrounds}
\DeclareGraphicsExtensions{.jpg,.pdf,.pdftex,.eps}

\newcommand{\ww}{\widetilde}
\newcommand{\Pb}{\mathbb{P}}

\newcommand{\LL}{\mathcal{L}}
\newcommand{\RR}{\mathcal{R}}

\newcommand{\dx}{\mathrm{d}}
\newcommand{\R}{\mathbb{R}}
\newcommand{\N}{\mathbb{N}}
\newcommand{\C}{\mathcal{C}}
\newcommand{\Z}{\mathbb{Z}}

\newcommand{\g}{g(\xi^{N},c)}

\newcommand{\GUE}{\mathrm{GUE}}
\newcommand{\GOE}{\mathrm{GOE}}

\newtheorem{tthm}{Theorem}

\newtheorem{prop}{Proposition}[section]
\def\note#1{\textup{\textsf{\color{blue}(#1)}}}

\newtheorem{cor}{Corollary}

\newtheorem{rem}[prop]{Remark}

\theoremstyle{definition}

\newcommand{\blocktheorem}[1]{%
  \csletcs{old#1}{#1}
  \csletcs{endold#1}{end#1}
  \RenewDocumentEnvironment{#1}{o}
    {\par\addvspace{1.5ex}
     \noindent\begin{minipage}{\textwidth}
     \IfNoValueTF{##1}
       {\csuse{old#1}}
       {\csuse{old#1}[##1]}}
    {\csuse{endold#1}
     \end{minipage}
     \par\addvspace{1.5ex}}
}

\usepackage{graphicx}
\usepackage{amssymb}
\usepackage{epstopdf}
\usepackage{nicefrac}
\begin{document}

\title{Limit profiles of ASEP }

\date{}

\maketitle 
\begin{abstract}

We study  the asymmetric simple exclusion process (ASEP) on a  segment $\{1,\ldots,b_N\}$  and are interested in its  total variation distance to equilibrium when started from an initial configuration $\xi^{N}$. We provide a general result which gives 
the cutoff window and profile whenever a KPZ-type limit theorem is available for  an extension of $\xi^{N}$ to $\Z$.  We apply this result to obtain the cutoff window and profile of ASEP on the segment with  flat, half-flat and step initial data. 
Our arguments are entirely probabilistic and make no use of  Hecke algebras.

\end{abstract}
\section{Introduction}
In this paper we consider the asymmetric simple exclusion process (ASEP) on a finite segment  $[a_N; b_N]=\{a_N,a_N+1,\ldots, b_N\}\subseteq \Z$ with an initial configuration $\xi^{N}$ that has $k_N$ particles.  
This is a continuous time Markov chain with state space
$$\Omega_{b_N, a_N,k_N}=\left\{\xi^{N}: [a_N; b_N]\to \{0,1\}, \sum_{i=a_N}^{b_N}\xi^{N}(i)=k_{N}\right\}.$$
We consider the $1's$ as particles and the $0's$ as holes. Each particle waits an exponential time with mean $1$ and then attempts to jump one step to the right with probability $p>1/2$ or to jump one step to the left with probability $q=1-p$. At the boundaries $a_N,b_N$, no particle can enter or exit. The jump attempt is successful if the target  position is occupied by a hole, see Section \ref{secgraph} for a more detailed construction. 
Note that by shifting $b_N$   we may take w.l.o.g. $a_N=1$ but it will be convenient to keep $a_N$ a free parameter.  ASEP on the segment $[a_N; b_N]$ with an initial data $\xi^{N}$ having $k_N$ particles has a unique invariant measure  which we denote by 
$\pi_{b_N,a_N, k_N}$.  For $a_N=1$ we set $\pi_{b_N, k_N}:=\pi_{b_N,1, k_N}. $We are interested how this process mixes to equilibrium and study the total variation (TV) distance 
\begin{equation}\label{TV}
d_{\xi^{N}}(t)=||P_{t}^{\xi^{N}}-\pi_{b_N, a_N, k_N}||_{\mathrm{TV}}:=\max_{A\subseteq \Omega_{b_N, a_N,k_N}}|P_{t}^{\xi^{N}}(A)-\pi_{b_N, a_N, k_N}(A)|,
\end{equation}
where $P_{t}^{\xi^{N}}$ is the law of ASEP started from $\xi^{N}$ at time $t$.  The maximal TV distance was studied previously, with the following result:

\begin{tthm}[Theorem 1 of \cite{BN22}]\label{mainold}Consider ASEP on $[1;N].$ Let $k=k_{N}$ with $\lim_{N\to\infty}k_{N}/N = \alpha$ and  $\alpha \in (0,1)$. Then 
we have for $c\in \R$ 
\begin{equation}\label{yesss}
\lim_{N\to \infty}\max_{\xi^{N} }d_{\xi^{N}}\left(\frac{(\sqrt{k_N}+\sqrt{N-k_N})^{2}+cN^{1/3}}{p-q}\right) =1-F_{\GUE}(cf(\alpha)), 
\end{equation}
where $f(\alpha)=\frac{(\alpha(1-\alpha))^{1/6}}{(\sqrt{\alpha}+\sqrt{1-\alpha})^{4/3}}$.
\end{tthm}
Here $F_{\GUE}$  is the cumulative distribution function  of the Tracy-Widom GUE distribution originating in  the theory of random matrices \cite{TW94}, namely  it is the limit law  of the rescaled largest eigenvalue of 
a matrix drawn from the Gaussian Unitary Ensemble (GUE). 
This result establishes max-TV cutoff for ASEP:  convergence to equilibrium is abrupt on the time scale $N$, and becomes smooth   in a  $N^{1/3}$ cutoff window with $1-F_\GUE$ being the  cutoff profile describing the smooth transition. 

In this paper, we are interested in a fixed sequence of initial data $\xi^{N}$. We provide a general result, Theorem \ref{main}, which  can  be applied to obtain a $N^{1/3}$ cutoff window as well as the cutoff profile. 
As we will see, novel cutoff profiles appear.  Since it is not too hard to see that instead of taking the maximum in \eqref{yesss}  we may take the initial data  $\xi^{N}=\mathbf{1}_{[1;k_N]}$ in \eqref{yesss}, 
we  actually recover Theorem \ref{mainold}, see Theorem \ref{thm3}. 

The main methodological difference between \cite{BN22} and the present paper is that \cite{BN22} utilized symmetries obtained from viewing multi-color  ASEP  as a random walk on a Hecke algebra.
This method is particularly useful  for studying $\mathbf{1}_{[1;k_N]}$ as initial data.   
Here, we use only probabilistic arguments - couplings and second class particles -  which allows to treat different initial data in a unified way.  Let us give three applications of our approach. 

First we consider initial data  where every second site is initially occupied, leading to a particle density $1/2$. 
\begin{tthm}\label{thm1}Consider ASEP on $[1;N]$ with initial configuration $\xi^{N}=\mathbf{1}_{2\Z\cap [1;N]}$. Then 
\begin{equation*}
\lim_{N\to\infty}d_{\mathbf{1}_{2\Z\cap [1;N]}}\left(\frac{N+cN^{1/3}}{p-q}\right)=1-F_{\GOE}(2^{-2/3}c).
\end{equation*}
\end{tthm}
Here $F_\GOE$ is the Tracy-Widom $\GOE$ distribution, again arising in random matrix theory \cite{TW94}. We thus see again cutoff with a $N^{1/3}$ window,  but with a different profile.  The restriction to density $1/2$ is not important, in Theorem \ref{thm4} of Section \ref{proofsec} we prove a more general statement for arbitrary densities. 

Next we consider so called half-flat initial data where initially every second negative integer is occupied.  

\begin{tthm}\label{thm2}Consider ASEP on $[-N/2;N/4]$ with initial configuration $\xi^{N}=\mathbf{1}_{2\Z\cap [-N/2;0]}$. Then 
\begin{equation*}
\lim_{N\to\infty}d_{\mathbf{1}_{2\Z\cap [-N/2;0]}}\left(\frac{N+cN^{1/3}}{p-q}\right)=1-\Pb(\mathcal{A}_{2\to1}(0)\leq 2^{-\frac{2}{3}}c).
\end{equation*}
\end{tthm}
Here $\mathcal{A}_{2\to1}$ is the $\mathrm{Airy}_{2\to 1}$ process introduced in \cite{BFS07}, whose one point distribution interpolates beteween $F_\GUE$ and $F_\GOE$. 
The choice of the interval $[-N/2;N/4]$  is not arbitrary.  For ASEP on $\Z$ started from $\mathbf{1}_{2\Z_{\leq 0}},$ at time $(N+cN^{1/3})/(p-q)$, the   particle initially located at (the even integer closest to) $-N/2$ and the hole initially  located at  $N/4$    will  have passed by $0$ with probability converging to $\Pb(\mathcal{A}_{2\to1}(0)\leq 2^{-\frac{2}{3}}c)$.  In the ASEP on $[-N/2;N/4]$   there are $k_N=N/4$ particles in the system,  and the system mixes when the  rightmost hole  (which starts at $N/4$) resp. the leftmost particle (which starts at $-N/2$) come close to $0$, i.e. at the position, where one observes the $\mathrm{Airy}_{2\to 1}$ process on $\Z$. 


Finally, we consider the initial data  $\xi^{N}=\mathbf{1}_{[1;k_N]}.$  Here we obtain the same mixing behaviour as for the max-TV distance. 
\begin{tthm}\label{thm3}Consider ASEP on $[1;N]$ with initial configuration $\xi^{N}=\mathbf{1}_{[1;k_N]}$ with $k_N /N\to \alpha\in (0,1)$. 
We have for $c\in \R$ 
\begin{equation}
\lim_{N\to \infty}d_{\xi^{N}}^{N, k_{N}}\left(\frac{(\sqrt{k_N}+\sqrt{N-k_N})^{2}+cN^{1/3}}{p-q}\right) =1-F_{\GUE}(cf(\alpha)), 
\end{equation}
where $f(\alpha)=\frac{(\alpha(1-\alpha))^{1/6}}{(\sqrt{\alpha}+\sqrt{1-\alpha})^{4/3}}$.
\end{tthm}
Since particles have a drift to the right, the equilibrium measure gives most of its mass to configurations  where the leftmost particle and rightmost hole are in $\mathcal{O}(1)$ distance to  $b_N-k_N$, see \eqref{CC} for a precise statement.  Since $\mathbf{1}_{[1;k_N]}$ places particles as much to the left as possible,  it is intuitive that $\xi^{N}$ will give the max-TV cutoff profile.

Theorems \ref{thm1}, \ref{thm2}, \ref{thm3} are proven in Section \ref{proofsec}, all proofs consist in applying Theorem \ref{main} below. 

\subsection{Main abstract Theorem} 
Here we provide our main abstract theorem which relates  the cutoff profile to the fluctuations of the  particles and holes  in an extension $\eta^{\infty}$ of $\xi^{N}$ to $\Z$.  
W.l.o.g. we work here with the segment $[1;b_N],$ we can always shift the interval and $\xi_N$  when needed. Given a particle configuration $\zeta: A\to\{0,1\},A\subseteq\Z,$ we call
\begin{align*}
&\LL(\zeta)=\inf\{i:\zeta(i)=1\}
\\&\RR(\zeta)=\sup\{i:\zeta(i)=0\}
\end{align*}
the (possibly infinite) position of the leftmost particle resp. rightmost hole.

Let $\xi^{N} \in \{0,1\}^{[1;b_N]} $ with $\sum_{i=1}^{b_N}\xi^{N}(i)=k_N$. We call an element $\eta^{\infty} \in \{0,1\}^{\Z}$ an extension of $\xi^{N}$ if
$$ \eta^{\infty}(j)=\xi^{N}(j), j\in [1;b_N].$$
If $\eta^{\infty}$ is an extension of $\xi^{N} $, we call 
\begin{align}\label{XH}
X_{k_{N}}:=\LL(\xi^{N}) \quad H_{b_{N}-k_{N}}:=\RR(\xi^{N})
\end{align}
the particle resp. hole of $\eta^{\infty} $ which is initially located at $\LL(\xi^{N})$ resp. $\RR(\xi^{N})$. 
We label the particles resp. holes of $\eta^{\infty}$  as
\begin{align}
&\cdots<X_{k_{N}+1}<X_{k_{N}}<X_{k_{N}-1}<\cdots\label{particleslab}
\\&\cdots<H_{b_N-k_{N}-1}<H_{b_N-k_{N}}<H_{b_N-k_{N}+1}<\cdots\label{holeslab}
\end{align}
and denote by $X_j (t)$ resp. $ H_i (t)$ the position at time $t$ of the particle $X_j$ resp. hole $H_i$ whenever the labels $i,j\in \Z$ have been assigned. 
We adopt the KPZ-type time scaling
\begin{equation}\label{g}
  g(\xi^{N},c) := \frac{D(\xi^{N})N + c N^{1/3}}{p-q},
  \qquad c\in\mathbb{R}
\end{equation}
where $D(\xi^{N})>0$  depends on $\xi^{N}$.
The following is our main generic theorem. 
\begin{tthm}\label{main}
Let $\xi^{N} \in \{0,1\}^{[1;b_N]} $ with   $\sum_{i=1}^{b_N}\xi^{N}(i)=k_N$ and $\delta,\kappa\in (0,1/3)$. We assume there is an extension $\eta^{\infty}$ of $\xi^{N},$ and a time scale $\g$ as in \eqref{g},  such that 
for $j\in \{-N^{\delta},0,N^{\delta}\}$  we have  with  $X_{k_{N}},H_{b_N -k_N}$  as in \eqref{XH} 
\begin{align}\label{particles}
&\lim_{N\to\infty}\Pb(X_{k_N+j}(\g)\leq b_N-k_N-N^{\kappa})=1-F(c)
\\&\lim_{N\to\infty}\Pb(H_{b_N-k_N+j}(\g)\geq b_N-k_N+N^{\kappa})=1-F(c),\label{holes}
\end{align}
where $F$ is a continuous probability distribution function.  Then 
\begin{equation}
\lim_{N\to\infty}d_{\xi^{N}}(\g)=1-F(c). 
\end{equation}
\end{tthm}
\begin{proof}
In Theorem \ref{lower_bound_thm2}, proven in Section \ref{seclower}, $1-F$ is established as lower bound  of $\lim_{N\to\infty}d_{\xi^{N}}(\g)$ under weaker assumptions.
In Corollary \ref{corfinal},  proven in Section \ref{secupper}, $1-F$ is established as upper bound, finishing the proof. 
\end{proof}

The preceding theorem allows us to compute the cutoff profile whenever we have a KPZ-type limit theorem available for $\eta^{\infty}$. By the known convergence of ASEP to the KPZ fixed point \cite{ACH24}, see Section \ref{onepoint} for the details, this is the case for a large class of configurations $\eta^{\infty}$. The choice of the extension will often be straightforward, in the cases of Theorems \ref{thm1}-\ref{thm3}, we may take $\mathbf{1}_{2\Z},\mathbf{1}_{2\Z_{\leq 0}},\mathbf{1}_{(-\infty;k_N]}.$  

Note that in general there is no logical relation between \eqref{particles} and \eqref{holes}, even though they are actually equivalent in the  applications we consider. 
To explain, take $j=0$ and consider for instance TASEP on $[1;10N]$ with particles initially located at $2\Z\cap[1;N] $ and $[N;10N-1],$ i.e. we have an isolated hole at position $10N,$ if it weren't for this hole, Theorem \ref{thm1} would apply (note  that in TASEP the particles move as if they were restricted $[1;N]$). Then one may check that for any extension  \eqref{particles} holds with $\g=N+cN^{1/3}$ and $F(c)=F_{\GOE}(2^{-2/3}c)$. However, for the system to mix, the rightmost hole, which performs a random walk,  needs to travel to position $N/2$ which will take about $9.5 N$   long  with fluctuations that are of order $N^{1/2}$ and asymptotically gaussian, thus \eqref{holes} holds at a different (larger) time scale with a different fluctuation exponent and different $F$, which dictates the mixing behaviour.  While we believe that such  situations could be included in our framework, we do not pursue this here. 
\subsection{Related Literature}
 
In the setting of Theorem \ref{main} and its applications  we always have  an abrupt convergence to equilibrium on scale $N$. In the language of Markov chains, this is an example of cutoff, which happens  at time $D(\xi^{N})N$, with a  $N^{1/3}$ window and limit (or cutoff) profile $1-F$, with $F$ being a one-point law of the KPZ fixed point,   the conjecturally universal limiting object of growth models in the Kardar-Parisi-Zhang (KPZ) universality class introduced in \cite{KQR16}. 

For earlier results on cutoff , we refer to the seminal papers \cite{A81}, \cite{DS81} as well as the review article \cite{D95}. 
The  very recent lecture notes   \cite{Sal1}  provide an overview on the cutoff phenomenon and its origins,  and explain  recently found general conditions (such as non-negative curvature) which guarantee cutoff. 

On the level of general results for profiles,  it was shown that for any continuous probability distribution function $F$  we have that $1-F$  may appear as limit profile, see \cite{Tey25}. Conditions establishing continuity of limit profiles  were recently given by \cite{N255}. 
Given this variety of possible limit profiles, a variety of different methods has been used to find them.  To give just a small selection of  recent examples, let us mention that the cutoff profile of random transpositions \cite{Tes20},  various reversible Markov chains \cite{NOT23} and the Bernoulli-Laplace urn \cite{SOS25} have been recently obtained. 

For ASEP,  the max-TV cutoff profile is known on a closed segment (this is the aforementioned Theorem \ref{mainold}), for  ASEP with one open boundary  \cite{HS22}  as well as the multicolor ASEP \cite{Z22}. Cutoff and Precutoff had been established earlier  in \cite{LL19} and  \cite{BBHM}, the paper  \cite{LL19} in particular also studied the mixing behavior of ASEP when started from an initial configuration $\xi^{N}$, and  in Theorem 3 of \cite{LL19} the constant $D(\xi^{N})$ is identified for general $\xi^{N}$.   For symmetric simple exclusion on the circle  the limit  profile is also known \cite{L16}.

All the results for limit profiles of  ASEP used, among other tools,  the integrability of ASEP in the form of Hecke algebras. In this paper, we rely instead  on probabilistic methods, in particular second class particles, to relate the limit profile to the limit law observed on $\Z$. 
It seems non-trivial to find such a relation using Hecke algebras.
Given that the convergence to the KPZ fixed point on $\Z$  is established for a large class of initial data in \cite{ACH24} (and is expected even under weaker assumptions than in \cite{ACH24}, see the remark after (2.10)  of \cite{ACH24})
we may thus obtain limit profiles fairly generally. 

Similar to \cite{BN22}, we seperately establish $1-F$ as lower and upper bound. For the lower bound, the task becomes more challenging as in \cite{BN22} as extending a general configuration $\xi^{N}$ to $\Z$ does not necessarily make particles move faster to the right (when $\xi^{N}=\mathbf{1}_{[1;k_N]}$ as in \cite{BN22} and we extend $\xi^{N}$ to $\Z$  as $\mathbf{1}_{-(\infty;k_N]}$ the particles in $[1;k_N]$ cannot slow down by the added particles in $(-\infty;0]$).  The lower bound is established  in Section \ref{seclower}. For the upper bound, we rely as \cite{BN22}  on hitting times, which were first used in \cite{BBHM} to study the mixing behaviour of ASEP,  however the link to the TV distance is again more involved as in \cite{BN22}. The key part is then to use second class particles to obtain the asymptotics of the  hitting time. This is done in Section \ref{secupper}.
For this, we need novel bounds on  the number of overtaking holes and second class particles, which are established in Section \ref{secovertake}.
\section{Preliminaries}\label{secgraph}
Let us start by briefly recalling the graphical construction of ASEP. Let $\mathcal{P}(z),z\in \Z,$ be a familiy of i.i.d. Poisson processes on $\R_+$ with rate $p$. 
Given any time point $t>0$, note that a.s. there is a double-infinite sequence  $(i_{n},n\in \Z)$ of integers 
$$\cdots<i_n<i_{n-1}<\cdots$$
such that $\mathcal{P}_{t}(i_{n})=0,n\in \Z,$ i.e. the Poisson processes $(\mathcal{P}(i_{n}),n\in \Z)$ all have made no jump up to time $t$. 
Given a configuration $\eta\in \{0,1\}^{\Z}$, we denote the   configuration  with updated edge $(z,z+1)$  by 
\begin{equation*}
\sigma_{z,z+1}(\eta)(i)=
\begin{cases}
\eta(i+1)  &\mathrm{for} \, i=z\\
\eta(i-1)  &\mathrm{for} \, i=z+1\\
\eta(i)   &\mathrm{else}.
\end{cases}
\end{equation*}
The update rule of ASEP is now as follows: We start from a configuration $\eta\in \{0,1\}^{\Z}$ at time $t=0$. We construct a Markov process $(\eta_{t}, t\geq 0)$  as follows:  Whenever $\mathcal{P}(z)$ has a jump at time $\tau$,  and $\eta_{\tau^{-}}(z)>\eta_{\tau^{-}}(z+1),$ we set $\eta_{\tau}:=\sigma_{z,z+1}(\eta_{\tau^{-}})$. 
If $\eta_{\tau^{-}}(z)<\eta_{\tau^{-}}(z+1),$ we toss an independent coin with success parameter $Q=q/p$, and with probability $Q$, $\eta_{\tau}:=\sigma_{z,z+1}(\eta_{\tau^{-}})$, and with probability $1-Q$, we set $\eta_{\tau}:=\eta_{\tau^{-}}.$
 Note that up to time $t$, we only need to apply these update rules inside each of the segments $[i_n;i_{n+1}]$ so that there is a well-defined first jump inside each such segment. Note also this construction works for any subset $A$ of consecutive integers, in all cases considered here, particles cannot enter or exit through the boundaries of  $A$.   The construction of ASEP on a segment $[a;b]$ is especially   easy: given an initial configuration $\xi$,   we may simply take Poisson processes $\mathcal{P}(a),\ldots, \mathcal{P}(b-1)$ and apply the update rules. 
\subsection{Basic coupling and second class particles}
Given two initial configurations $\eta,\mu$, we can couple the processes $(\eta_{t}, t\geq 0),(\mu_{t}, t\geq 0)$ by constructing them with the same family of Poisson processes  $\mathcal{P}(z),z\in \Z$. This coupling is called the basic coupling. Note we can also  couple   ASEPs on different subsets on $\Z$ this way.

Let $\mu,\eta\in \{0,1\}^{\Z}$ be such that $\mu(i)\leq \eta(i),i\in\Z,$ we write $\mu\leq \eta$ as shorthand for this.  Then, under the basic coupling, $\mu_t \leq \eta_t$ for all $t\geq 0$. We then say that a site $i$ is occupied by a second class particle at time $t$  if $\mu_t (i)<\eta_t (i)$, 
and we say that $i$ is occupied by a first class particle  if $\mu_{t}(i)=\eta_{t}(i)=1.$ We write the first class particle as $1's$, the second class particles as $2's$ and the holes as $0's$. Setting $\zeta_t := 2*\eta_t-\mu_{t}$ we obtain an ASEP $(\zeta_t, t\geq 0)$ on $\{0,1,2\}^{\Z}$. 

 It follows from the construction that the $2's$ treat the $0's$ as holes, and are treated as holes by the $1's$.  More formally,  whenever $\{\zeta_{\tau^{-}}(z),\zeta_{\tau^{-}}(z+1)\}=\{0,2\}$ we update the edge $(z,z+1)$  by the same rule as before. 
If   $\eta_{\tau^{-}}(z)=1, \eta_{\tau^{-}}(z+1)=2$ we update the process as $\zeta_{\tau}:=\sigma_{z,z+1}(\eta_{\tau^{-}}),$  and if $\zeta_{\tau^{-}}(z)=2, \zeta_{\tau^{-}}(z+1)=1,$ we toss an independent coin with success parameter $Q=q/p$, and with probability $Q$, $\zeta_{\tau}:=\sigma_{z,z+1}(\zeta_{\tau^{-}})$, and with probability $1-Q$, we set $\zeta_{\tau}:=\zeta_{\tau^{-}}.$ 

\subsection{One point limit laws of ASEP}\label{onepoint}
KPZ-type one point limit laws for ASEP were first established for Step \cite{TW08b},  Step-Bernoulli \cite{TW09b} and stationary  \cite{A18} initial data.
The convergence of ASEP (in fact, several coupled ASEPs)  to the KPZ fixed point (which includes the convergence of one point distributions as special case) was recently established 
in \cite{ACH24} for general initial conditions  which are bounded below lines of arbitrary slope at $\pm\infty$. 
 
In the following, we quickly review the convergence result Theorem 2.9 of \cite{ACH24} specialized to the  setting in which we make use of it.  We refer to Section 2 of \cite{ACH24}  for the details. 
Given an initial particle configuration $\eta_0$, associate to it a height function $h_0$ by setting
$$ \eta_{0}(x)=h_0 (x-1) - h_{0}(x), \quad h_{0}(0):=0. $$
Given an initial height profile $h_0$ we define the height function $h(h_0,y,t)$ as follows: Whenever a particle jumps from $y$ to $y+1$ at a time $r$,   we update $h(h_0,y,r):=h(h_0,y,r^{-})+1$ and $h(h_0,x,r)=h(h_0,x,r^{-})$ for $x\neq y,$ 
whereas when a particle jumps from $y$ to $y-1$ at time $r$, we update $h(h_0,y-1,r):=h(h_0,y-1,r^{-})-1$ and $h(h_0,x,r)=h(h_0,x,r^{-})$ for $x\neq y-1$. 

 Let $\rho\in (0,1)$  and set $2^{5/3}(\rho(1-\rho))^{1/3}=:\kappa(\rho).$\footnote{The variables $\varepsilon,\alpha,\mu(\alpha),\mu^{'}(\alpha),\beta(\alpha)$ from  page 9 of \cite{ACH24} relate to ours as $\varepsilon=1/N,\alpha=1-2\rho,\mu(\alpha)=\rho^{2},\mu^{'}(\alpha)=-\rho, \beta(\alpha)=\kappa(\rho).$}
 Given a sequence of initial height functions $h_{0}^{1/N}$ we set
\begin{equation}\label{h1}\mathfrak{h}^{1/N}_{0}(x):=2^{-1/3}(\rho(1-\rho))^{-2/3}N^{-1/3}(- \rho \kappa(\rho) x N^{2/3} - h_{0}^{1/N}( \kappa(\rho) x N^{2/3})).\end{equation}
Let $\gamma=p-q$, which is the drift of a single-particle (in \cite{ACH24}, this drift is $1-q$, but this simply amounts to   a rescaling of time). 
We then define the rescaled height function 
\begin{equation}
\begin{aligned}\label{h2}
\mathfrak{h}^{1/N}(\mathfrak{h}^{1/N}_{0}, x,t)=\frac{(\rho(1-\rho))^{-2/3}}{(2N)^{1/3}}&\bigg(2\rho^{2}tN-\rho\kappa(\rho) xN^{2/3} \\&-h(h_{0}^{1/N},2(1-2\rho)Nt+\kappa(\rho)xN^{2/3}, 2 Nt/\gamma)\bigg).
\end{aligned}
\end{equation}
The following gives the convergence of one-point distributions of ASEP. For the full  and much stronger original statement, see Theorem 2.9 of \cite{ACH24}. For the definition of the KPZ fixed point, see Definition 2.8  of \cite{ACH24}, for local convergence in UC,  see page 137 of \cite{KQR16}. 
\begin{tthm}[Special case of Theorem 2.9 of \cite{ACH24}]\label{thm29}
Let $\mathfrak{h}^{1/N}_{0},N\geq 1$ be a sequence of rescaled initial height functions as in \eqref{h1}. Assume that $\mathfrak{h}^{1/N}_{0}\to \mathfrak{h}_{0}$ locally in UC with $\mathfrak{h}_{0}:\R\to \R\cup\{-\infty\}$ a UC function,  and that  for a $\lambda>0$ we have for all $x;N$ 
\begin{equation*}
\mathfrak{h}_{0}^{1/N}(x)\leq C(1+|x|) \quad \sup_{x\in [-\lambda,\lambda]}\mathfrak{h}_{0}^{1/N}(x) \geq -\lambda. 
\end{equation*}
Then $\mathfrak{h}^{1/N}(\mathfrak{h}^{1/N}_{0}, x,t)$ converges in distribution to $\mathfrak{h}(\mathfrak{h}_0, x,t),$   the KPZ fixed point started from $\mathfrak{h}_0$  and evaluated at space point $x$ and time point $t$. 
\end{tthm}

\subsection{Censoring Inequality}
We use the following partial order for configurations $\eta,\zeta$ with $\LL(\eta),\LL(\zeta)>-\infty$ :
\[
   \zeta \preceq \eta
  \quad\Longleftrightarrow\quad
  \sum_{i=-\infty}^r \eta(i) \;\le\; \sum_{i=-\infty}^r \zeta(i)
  \quad \text{for all } r\in\mathbb{Z},
\]
meaning that $\eta$ places (weakly) more particles to the right than $\zeta$.
ASEP is \emph{attractive} with respect to this order: if $\eta_0 \succeq \zeta_0$,  then 
$\eta_t \succeq \zeta_t$ for all $t\ge0$ under the basic coupling. Clearly, if $\eta_0 \succeq \zeta_0$, then $\LL(\eta_0)\geq \LL(\zeta_0)$. 

Let $(\eta_t,t\geq 0)$ be an ASEP. A censoring scheme $\C$ is a random cadlag function, independent of $(\eta_t,t\geq 0), $ which maps 
$$\mathcal{C}: \R_{\geq 0} \to \mathcal{P}(E),$$ where $\mathcal{P}(E)$ is the power set of edges on which the process $(\eta_t,t\geq 0)$ is defined. We call the edges $(i,i+1)\in \mathcal{C}(t)$ censored. We then define a censored process $(\eta_{t,\mathcal{C}},t\geq 0)$
by imposing that   a transition along an edge $e$ in the uncensored process $(\eta_t,t\geq 0)$  
at time $\ell$ is also  performed 
in the censored process happens  iff $e \notin \C(\ell),$ and no other transitions are performed in $(\eta_{t,\mathcal{C}},t\geq 0)$.
The following Proposition is Remark 2.13 of \cite{GNS20}, here we define for $Z\in \Z$  the ASEP  $ ( \eta^{-\mathrm{step}(Z)}_{\ell}, \ell\geq 0) $ started from the reversed step initial data
\begin{equation*}
\eta^{-\mathrm{step}(Z)}(j)=\mathbf{1}_{\Z_{\geq Z}}(j),
\end{equation*}
for $Z=0$ we simply write $\eta^{-\mathrm{step}}$.
\begin{prop}[Remark 2.13 of \cite{GNS20}]\label{censor}
Let $\C$ be a censoring scheme for $(\eta^{-step(Z)}_t , t\geq 0)$. Then there is a coupling of $(\eta^{-step(Z)}_{t},t\geq 0)$ and $(\eta^{-step(Z)}_{t,\C},t\geq 0)$ such that $$\Pb(\eta^{-step(Z)}_{t}\preceq
\eta^{-step(Z)}_{t,\C})=1.$$ In particular, for all $x\in \Z$, we have 
$$\Pb(\LL(\eta^{-step(Z)}_{t})\leq x)\geq \Pb(\LL(\eta^{-step(Z)}_{t,\C})\leq x).$$
\end{prop}

\section{Proof of Theorems \ref{thm1}, \ref{thm2}, \ref{thm3}.}\label{proofsec}
Here we apply Theorem \ref{main} to prove Theorems \ref{thm1}, \ref{thm2}, \ref{thm3}.  As input, we use the convergence of one-point distributions to the KPZ fixed point, explained in Section \ref{onepoint}, for Theorem \ref{thm3} we directly use the result of \cite{TW08}, in the form of Corollary 1 of \cite{BN22}, which gives exactly  the convergence \eqref{particles}, 

We start by showing the following generalization\footnote{With \eqref{rho} being obviously  satisfied for $\rho=1/2$ for $\xi^{N}$ of Theorem \ref{thm1}. } 
of Theorem \ref{thm1}.
\begin{tthm}\label{thm4}Consider ASEP on $[1;N]$ with initial configuration $\xi^{N}$ such that there  is a $\rho\in (0,1)$ and a function $R(N)$ with $R(N)N^{-1/3}\to0$ such that for all $0\leq a<b\leq 1$ we have 
\begin{equation}\label{rho}\left| \sum_{i=aN}^{bN}\xi^{N}(i)-\rho(b-a)N\right|\leq R(N).\end{equation}
Then 
$$\lim_{N\to\infty}d_{\xi^{N}}\left(\frac{N+cN^{1/3}}{p-q}\right)=1-F_{\GOE}((\rho(1-\rho))^{1/3}c). $$
\end{tthm}
\begin{proof}
We apply Theorem \ref{main}. 
Throughout the proof, $R_{i}(N),i \geq 1$  will denote a sequence with $R_{i}(N)N^{-1/3}\to 0$ whose concrete values are immaterial. 
As extension we may take any configuration which satisfies \eqref{rho} for all $a,b\in \R, a<b$, we may e.g. paste $\xi^{N}$ in each interval $[kN,(k+1)N]$ for $k\in \Z\setminus{\{0\}}$ to obtain such a $\eta^{\infty}.$ 

Note that by assumption we have  $k_N=\rho N+R_1(N)$ and there are $\rho(1-\rho)N+R_2(N)$ particles in $[1;b_N-k_N]$.   Take $t=1-cN^{-2/3}$ such that $tN=N+cN^{1/3}$. Note we may shift the interval $[1;b_N]$  such that $b_N-k_N-N^{\kappa}=(1-2\rho)Nt$ .
Let $h_{0}^{1/N}$ be the height function of $\eta^{\infty}$ and note that since \eqref{rho} holds for $a,b\in \R$, $\mathfrak{h}_{0}^{1/N}\to 0$ locally in UC.  The event 
$$ X_{k_N+j}(\g)\leq b_N-k_N-N^{\kappa}$$
is the same as the event  (recall $\gamma=p-q$ and thus $Nt/\gamma=\g$) 
\begin{equation}\label{yes}
h(h_{0}^{1/N}, (1-2\rho)Nt, Nt/\gamma)-h_{0}^{1/N}((1-2\rho)Nt)\leq \rho(1-\rho)N+R_2(N).
\end{equation}

Since $h_{0}^{1/N}((1-2\rho)Nt)=-\rho(1-2\rho)Nt+R_3(N)$ we see that \eqref{yes} is equivalent to 
\begin{equation}\label{yes2}
h(h_{0}^{1/N}, (1-2\rho)Nt, Nt/\gamma) -\rho^{2}Nt \leq - \rho(1-\rho)cN^{1/3}+R_4(N)
\end{equation}
which leads to 
\begin{equation}\label{yes3}
 2^{-1/3}(\rho(1-\rho))^{1/3}c+R_5(N)N^{-1/3}\leq \frac{\rho^{2}Nt-h(h_{0}^{1/N}, (1-2\rho)Nt, Nt/\gamma)}{N^{1/3}2^{1/3}(\rho(1-\rho))^{2/3}}.
\end{equation}
Now the r.h.s. of \eqref{yes3} equals $\mathfrak{h}^{1/N}(\mathfrak{h}_{0}^{1/N},0,t/2)$  and thus  converges by Theorem \ref{thm29} in law to $\mathfrak{h}(0,0,1/2),$ the KPZ fixed point started from flat initial data evaluated at $x=0$ and time $1/2$. We thus obtain 
 \begin{align}
 &\lim_{N\to\infty}\Pb(X_{k_N+j}(\g)\leq b_N-k_N-N^{\kappa})
 \\&=\lim_{N\to\infty}\Pb\left(\frac{(\rho(1-\rho))^{1/3}}{2^{1/3}}c \leq \frac{\rho^{2}Nt-h(h_{0}^{1/N}, (1-2\rho)Nt, Nt/\gamma)}{N^{1/3}2^{1/3}(\rho(1-\rho))^{2/3}}\right)
 \\&=\Pb\left(\frac{(\rho(1-\rho))^{1/3}}{2^{1/3}}c \leq \mathfrak{h}(0,0,1/2)\right)\label{p}
 \\&=1-F_{\GOE}((\rho(1-\rho))^{1/3}c). \label{p2}
 \end{align}
 In the last step, we used that the KPZ fixed point started from $0$ is the $\mathrm{Airy}_1$  process with $\GOE$ one-point distribution (see \cite{KQR16},  page 160).
 
 Finally, to prove \eqref{holes}, it suffices to note that there are $\rho(1-\rho)N+R_{5}(N)$ many holes present initially in $[b_N-k_N,b_N],$ hence  
$$  H_{b_N-k_N+j}(\g)\geq b_N-k_N+N^{\kappa}$$ is equivalent  to 
\begin{equation}\label{yes4}
h(h_{0}^{1/N}, (1-2\rho)Nt+2N^{\kappa}, Nt/\gamma)-h_{0}^{1/N}((1-2\rho)Nt+2N^{\kappa})\leq \rho(1-\rho)N+R_5(N),
\end{equation}
which is just \eqref{yes} up to the asymptotically irrelevant shift by $2N^{\kappa}$ (note $\kappa<1/3$).  Hence   \eqref{holes} follows by the same proof as before. 
\end{proof}
\begin{proof}[Proof of Theorem \ref{thm2}]
We again apply Theorem \ref{main}. As extension we take $\eta^{\infty}=\mathbf{1}_{2\Z_{\leq 0}}$. Then $\mathfrak{h}^{1/N}_{0}$ converges locally in UC to the UC function $\mathfrak{h}^{hf}=-\infty \mathbf{1}_{(0,+\infty)}$ with $\rho=1/2$. 
We have $k_{N}=N/4, b_N-k_N=0$. As in the proof of Theorem \ref{thm4}, by  $R_{i}(N) $ we denote a sequence with $R_{i}(N) N^{-1/3}\to 0$.
Then
$$ X_{k_N+j}(\g)\leq b_N-k_N-N^{\kappa}$$
is the same as the event 
\begin{equation}\label{yes5}
h(h_{0}^{1/N},R_6 (N),( N+cN^{1/3})/\gamma ) \leq N/4 +R_7 (N). 
\end{equation}
Setting $\tilde{N}=N+cN^{1/3}$ we obtain 
\begin{equation}\label{yes17}
-h(h_{0}^{1/N} , R_{8}(\tilde{N}),\tilde{N}/\gamma ) +\tilde{N}/4 \geq c\tilde{N}^{1/3}/4+ R_{9}(\tilde{N}). 
\end{equation}
Applying to \eqref{yes5} the rescaling \eqref{h2} with $x=0,\rho=1/2$, we obtain as $N\to \infty$ limit of the l.h.s. of \eqref{yes17}  the KPZ fixed point started from $\mathfrak{h}^{hf}$ and evaluated at  $x=0,t=1/2,$ which  is   $\mathcal{A}_{2\to 1}(0)$ (see \cite{KQR16},  page 161), establishing  \eqref{particles}.  The statement \eqref{holes} follows in exactly the same way, note there are $N/4$ holes present initially in $[1;N/4]$ hence establishing \eqref{holes} is equivalent to establishing the convergence of the probability of the event  \eqref{yes5}, finishing the proof. 
\end{proof}
\begin{proof}[Proof of Theorem \ref{thm3}]
We may shift the interval $[1;N]$ to the left and consider $[-k_N;N-k_N+1]$ instead. As extension we take the step initial data $\mathbf{1}_{\Z_{\leq 0}}$.
The convergence \eqref{particles} is then exactly (12) of \cite{BN22} (which itself is a corollary of \cite{TW08}), whereas  the statement \eqref{holes} follows from particle-hole duality. 
\end{proof}
\section{Proof of the lower bound}\label{seclower}
Here we  show that $1-F$ is a lower bound to the cutoff profile. For this, it actually suffices to make an assumption about $X_{k_N}$, as the next Theorem shows:

\begin{tthm}\label{lower_bound_thm2}
Fix \(c \in \mathbb{R}\).  
Assume there is an extension $\eta^{\infty}$ and a time scale $\g$ as in \eqref{g} such that, with $X_{k_{N}}$ as in \eqref{XH}, we have  for some \(\delta>0\),
\[
  \lim_{N\to\infty} 
  \mathbb{P}\!\Big(
    X_{k_N}\big(\g\big) 
    \le 
    b_N - k_N - N^{\delta}
  \Big)
  = 1 - F(c),
\]
with  $F$  a continuous probability distribution function.
Then,
\[
  \liminf_{N\to\infty} 
  d_{k_N}\!\big(\g\big)
  \;\ge\;
  1 - F(c).
\]
\end{tthm} To prove Theorem \ref{lower_bound_thm2}, we need to know that $ \mathcal{L}(\xi^N_{\g}) 
    $ does not overtake $
    X_{k_N}(\g) $  by an asymptotically relevant amount.  This is the content of the following Theorem. 
\begin{tthm}\label{lower_bound_thm1}
Let $ \delta>0$ and $ \eta^\infty$ be an extension of  $\xi^N$ as in Theorem \ref{lower_bound_thm2}.  Then 
\[
  \lim_{N\to\infty}
  \mathbb{P}\!\left(
    \mathcal{L}(\xi^N_{\g}) 
    \le 
    X_{k_N}(\g) + N^{\delta}
  \right) = 1.
\]
\end{tthm}
Given Theorem \ref{lower_bound_thm1}, we can prove Theorem \ref{lower_bound_thm2}:
\begin{proof}[Proof of Theorem  \ref{lower_bound_thm2} assuming Theorem \ref{lower_bound_thm1}]
For \(l \ge 0\), define
\[
A_{b_{N}}(l)
  = \Bigl\{
    v \in \{0,1\}^{[1,b_N]} :
    \mathcal{L}(v) < b_N - k_N - l
  \Bigr\}.
\]
By Proposition~4.2 of \cite{BN22}, there exist constants 
\(C_1, C_2 > 0\) such that
\[
  \pi_{b_N,k_N}\bigl(A_{b_N}(l)\bigr)
  \le C_{1} e^{-C_{2} l}.
\]
Hence,
\begin{align}
\lim_{N\to\infty} d_{k_N}\!\big(\g\big)
&\ge
\lim_{N\to\infty}
\Bigl[
  \mathbb{P} \Big(
    \xi^{N}_{\g} 
    \in A_{b_{N}}(N^{\delta})
  \Big)
  - 
  \pi_{b_N,k_N}\bigl(A_{b_N}(N^{\delta})\bigr)
\Bigr] \label{eq:thm2_1}\\[6pt]
&=
\lim_{N\to\infty}
\Bigl[
  \mathbb{P}\Big(
    \mathcal{L}\!\big(\xi^{N}_{\g}\big)
    \le b_N - k_N - N^{\delta}
  \Big)
  - 
  \pi_{b_N,k_N}\bigl(A_{b_N}(N^{\delta})\bigr)
\Bigr] \label{eq:thm2_2}\\[6pt]
&\ge
\lim_{N\to\infty}
\mathbb{P}\Big(
  \mathcal{L}\big(\xi^{N}_{\g}\big)
  \le b_N - k_N - N^{\delta} 
\Big)
- C_{1} e^{-C_{2} N^{\delta}} \label{eq:thm2_3}\\[6pt]
&\ge
\lim_{N\to\infty}
\mathbb{P} \Big(
  X_{k_N}\big(\g\big)
  \le b_N - k_N - 2N^{\delta}
\Big)
\quad\text{(by Theorem~\ref{lower_bound_thm1})} \label{eq:thm2_4}\\[6pt]
&= 1 - F(c)\quad\text{(by assumption)} \label{eq:thm2_5}
\end{align}
\end{proof}


\subsection{Auxiliary Constructions}

To compare the dynamics of the finite system $\xi^N$ with those of its infinite
extension $\eta^\infty$, we construct a sequence of intermediate configurations
\[
  \eta^{\infty,1},\quad 
  \eta^{\infty,2},\quad 
  \eta^{\infty,3},\quad 
  \eta^{\infty,4}
\]
each obtained from the previous one by a localized modification and/or removing a boundary of the configuration. See Figure 1 for an illustration. 
All processes evolve under the \emph{basic coupling}. Let us start by defining $\eta^{\infty,1}$.

\begin{prop}\label{prop:eta1_xi_N}
Let $\eta^{\infty}$ be the extension from Theorem \ref{lower_bound_thm2}.
Define an extension to positive integers $\eta^{\infty,1}\in \{0,1\}^{\Z_{\geq 1}}$ of \(\xi^N\) by
\[
\eta^{\infty,1}(i) :=
\begin{cases}
  \xi^N(i), & i \in [1,b_N],\\[3pt]
  0, & i \in (b_N,\, b_N + N^{\delta}],\\[3pt]
  \eta^\infty(i - N^{\delta}), & i > b_N + N^{\delta}.
\end{cases}
\]
Then
\[
  \lim_{N\to \infty}\mathbb{P}\!\left(
    \mathcal{L}\big(\eta^{\infty, 1}_{\,g(\xi^{N},c)}\big)
    \;\ge\;
    \mathcal{L}\big(\xi^{N}_{\,g(\xi^{N},c)}\big)
  \right) = 1.
\]
\end{prop}

\begin{proof}

If \(\eta^{\infty,1}\) has no particle to the right of \(b_N\), then on \([1,b_N]\) it
coincides with \(\xi^N\) except that the right boundary is removed and the statement follows. Otherwise set
\[
  X^{*}_{0} := \min\{i \geq b_N + N^{\delta} : \eta^{\infty,1}(i) = 1\},
\]
and consider the reversed step initial data \(\eta^{-\mathrm{step}(X^{*}_{0})}=\mathbf{1}_{\Z_{\geq X^{*}(0)}}\).

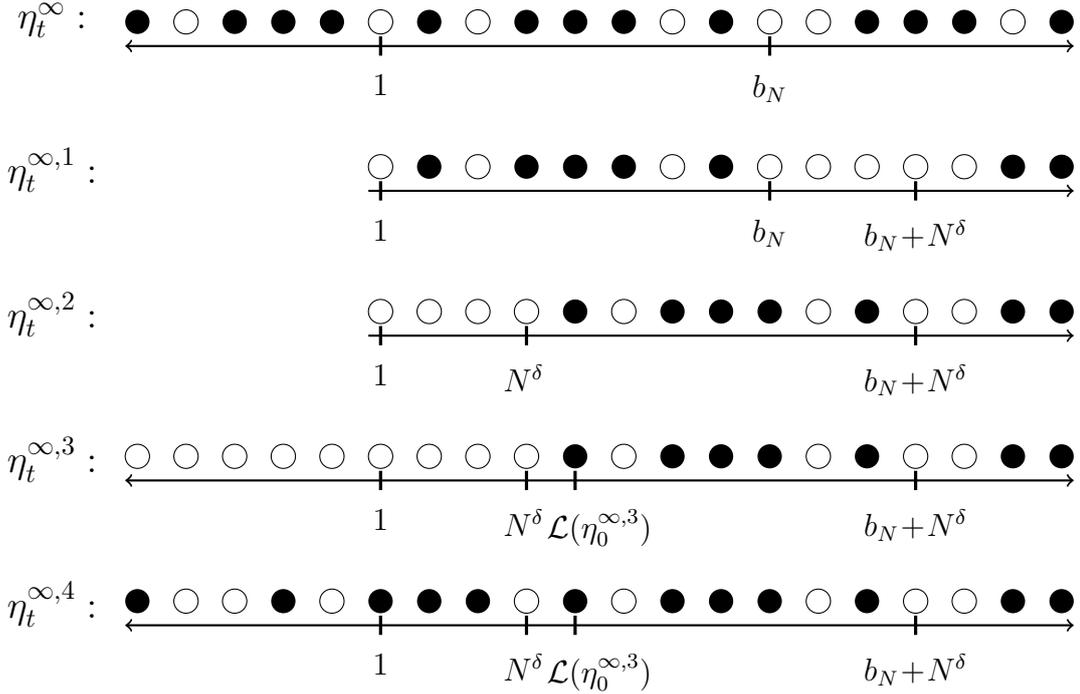
\begin{figure}[H]\begin{center}\label{figure1}
\begin{tikzpicture}[scale=1.6]

\draw[thick,<->] (-2.1,0) -- (5.7,0);
\node[below] at (-2.7,0.45) {\large$\eta^{\infty}_t:$};

\draw[very thick] (0.0,0.08)--(0.0,-0.08);
\draw[very thick] (3.2,0.08)--(3.2,-0.08);
\node[below] at (0.0,-0.15) {$1$};
\node[below] at (3.2,-0.15) {$b_N$};

\foreach \x in {-1.6}
  \draw (\x,0.2) circle (0.1);
\foreach \x in {-1.2,-0.8,-0.4, -2.0}
  \fill (\x,0.2) circle (0.1);

\foreach \x in {0.0,0.8,2.4,3.2}
  \draw (\x,0.2) circle (0.1);
\foreach \x in {0.4,1.2,1.6,2.0,2.8}
  \fill (\x,0.2) circle (0.1);

\foreach \x in {3.6, 5.2}
  \draw (\x,0.2) circle (0.1);
\foreach \x in {4.0,4.4,4.8,5.6}
  \fill (\x,0.2) circle (0.1);

\begin{scope}[yshift=-1.2cm]
\draw[thick,->] (-0.1,0) -- (5.7,0);
\node[below] at (-2.7,0.45) {\large$\eta^{\infty,1}_t:$};

\draw[very thick] (0.0,0.08)--(0.0,-0.08);
\draw[very thick] (3.2,0.08)--(3.2,-0.08);
\draw[very thick] (4.4,0.08)--(4.4,-0.08);
\node[below] at (0.0,-0.15) {$1$};
\node[below] at (3.2,-0.15) {$b_N$};
\node[below] at (4.4,-0.15) {$b_N\!+\!N^\delta$};

\foreach \x in {0.0,0.8,2.4,3.2}
  \draw (\x,0.2) circle (0.1);
\foreach \x in {0.4,1.2,1.6,2.0,2.8}
  \fill (\x,0.2) circle (0.1);

\foreach \x in {3.6,4.0, 4.4}
  \draw (\x,0.2) circle (0.1);

\foreach \x in {4.8}
  \draw (\x,0.2) circle (0.1);
\foreach \x in {5.2, 5.6}
  \fill (\x,0.2) circle (0.1);
\end{scope}

\begin{scope}[yshift=-2.4cm]
\draw[thick,->] (-0.1,0) -- (5.7,0);
\node[below] at (-2.7,0.45) {\large$\eta^{\infty,2}_t:$};

\draw[very thick] (0.0,0.08)--(0.0,-0.08);
\draw[very thick] (1.2,0.08)--(1.2,-0.08);
\draw[very thick] (4.4,0.08)--(4.4,-0.08);
\node[below] at (0.0,-0.15) {$1$};
\node[below] at (1.2,-0.15) {$\!N^\delta$};
\node[below] at (4.4,-0.15) {$b_N\!+\!N^\delta$};

\foreach \x in {0.0,0.4,0.8}
  \draw (\x,0.2) circle (0.1);

\foreach \x in {0.0,0.8,2.4,3.2}
  \draw (\x+1.2,0.2) circle (0.1);
\foreach \x in {0.4,1.2,1.6,2.0,2.8}
  \fill (\x+1.2,0.2) circle (0.1);

\foreach \x in {4.4}
  \draw (\x+0.4,0.2) circle (0.1);
\foreach \x in {4.8,5.2}
  \fill (\x+0.4,0.2) circle (0.1);
\end{scope}

\begin{scope}[yshift=-3.6cm]
\draw[thick,<->] (-2.1,0) -- (5.7,0);
\node[below] at (-2.7,0.45) {\large$\eta^{\infty,3}_t:$};

\draw[very thick] (0.0,0.08)--(0.0,-0.08);
\draw[very thick] (1.2,0.08)--(1.2,-0.08);
\draw[very thick] (1.6,0.08)--(1.6,-0.08);

\draw[very thick] (4.4,0.08)--(4.4,-0.08);
\node[below] at (0.0,-0.15) {$1$};
\node[below] at (1.2,-0.15) {$\!N^\delta$};
\node[below] at (1.8,-0.15) {$ \LL(\eta^{\infty,3}_0)$};
\node[below] at (4.4,-0.15) {$b_N\!+\!N^\delta$};

\foreach \x in {-2.0,-1.6,-1.2,-0.8,-0.4}
  \draw (\x,0.2) circle (0.1);

\foreach \x in {0.0,0.4, 0.8}
  \draw (\x,0.2) circle (0.1);

\foreach \x in {0.8,1.6,3.2}
  \draw (\x+0.4,0.2) circle (0.1);
\foreach \x in {1.2,2.0,2.4,2.8,3.6}
  \fill (\x+0.4,0.2) circle (0.1);

\foreach \x in {4.4, 4.8}
  \draw (\x,0.2) circle (0.1);
\foreach \x in {5.2, 5.6}
  \fill (\x,0.2) circle (0.1);
\end{scope}

\begin{scope}[yshift=-4.8cm]
\draw[thick,<->] (-2.1,0) -- (5.7,0);
\node[below] at (-2.7,0.45) {\large$\eta^{\infty,4}_t:$};
\node[below] at (1.8,-0.15) {$ \LL(\eta^{\infty,3}_0)$};
\draw[very thick] (0.0,0.08)--(0.0,-0.08);
\draw[very thick] (1.2,0.08)--(1.2,-0.08);
\draw[very thick] (4.4,0.08)--(4.4,-0.08);
\node[below] at (0.0,-0.15) {$1$};
\draw[very thick] (1.6,0.08)--(1.6,-0.08);
\node[below] at (1.2,-0.15) {$\!N^\delta$};
\node[below] at (4.4,-0.15) {$b_N\!+\!N^\delta$};

\foreach \x in {-2.8,-2.4,-1.6}
  \draw (\x+1.2,0.2) circle (0.1);
\foreach \x in {-3.2,-2.0,-1.2,-0.8,-0.4}
  \fill (\x+1.2,0.2) circle (0.1);

\foreach \x in {0.0,0.8,2.4,3.2}
  \draw (\x+1.2,0.2) circle (0.1);
\foreach \x in {0.4,1.2,1.6,2.0,2.8}
  \fill (\x+1.2,0.2) circle (0.1);

\foreach \x in {3.6}
  \draw (\x+1.2,0.2) circle (0.1);
\foreach \x in {4.0,4.4}
  \fill (\x+1.2,0.2) circle (0.1);
\end{scope}

\end{tikzpicture}\end{center}

\caption{
From top to bottom:   $\eta^{\infty,1}_t$ is obtained from $\eta^{\infty}_t$ by considering only the configuration to the right of $1$ and by inserting  $N^{\delta}$ holes to the right of $b_N$. In  $\eta^{\infty,2}_t$ these inserted holes are moved   to the right of $1$. In  $\eta^{\infty,3}_t$ we extend the configuration to $\Z$ by filling all non-positive integers with holes,   and in  $\eta^{\infty,4}_t$ reintroduce the particles that were to to the left of $1$  in  $\eta^{\infty,1}$, so that $\eta^{\infty,4}$ is just $\eta^{\infty}$ shifted by $N^{\delta}$. 
By comparing $\eta^{\infty,i}$ with  $\eta^{\infty,i+1}$
we show that  under the basic coupling $ 
\Pb(\LL(\xi^{N}(\g)\leq\LL(\eta^{\infty,4}_\g)),$ converges to $1$.  Note that when going from $\eta^{\infty,3}$ to $\eta^{\infty,4}$  we add particles, so the particle initially at $ \LL(\eta^{\infty,3}_0)$  in $\eta^{\infty,4}$  is never to the left of  the leftmost particle of $ \eta^{\infty,3}.$ 
This provides the comparison needed in the proof of Theorem~\ref{lower_bound_thm1}. 
} 
\label{fig:five-eta-configs}
\end{figure}

Since \(\eta^{-\mathrm{step}(X^{*}_{0})}\) fills every site to the right of \(X^{*}_{0}\) while
\(\eta^{\infty,1}\) may have holes there, we have
\[
X^{*}_{0}(t)\ \ge\ \mathcal{L}\big(\eta^{-\mathrm{step}(X^{*}_{0})}_{t}\big).
\]
For \(p>q\), the reversed step has rightward drift with exponential left-tail bounds
for its leftmost particle; thus there exist \(C_1,C_2>0\) such that, for all large \(N\) (see Proposition~\ref{block}),
\[
  \mathbb{P}\!\Big(
    \inf_{0\le s \leq g(\xi^N,c)}
    \mathcal{L}\big(\eta^{-\mathrm{step}(X^{*}_{0})}_s\big)
    \geq b_N + N^{\delta} - (g(\xi^N,c))^{\delta/2}
  \Big) \ge 1- C_1 e^{-C_2 (g(\xi^N,c))^{\delta/2}}.
\]
Since \(b_N + N^{\delta} - (g(\xi^N,c))^{\delta/2} > b_{N} + N^{\delta/2}\) for \(N\) large,
\[
  \mathbb{P}\!\Big(
    \inf_{0\le s \leq g(\xi^N,c)}
    \mathcal{L}\big(\eta^{-\mathrm{step}(X^{*}_{0})}_s\big)
    \geq  b_{N} + N^{\delta/2} \Big) \ge 1- C_1 e^{-C_2 (g(\xi^N,c))^{\delta/2}}.
\]
On this event, particles initially to the right of \(b_N\) in \(\eta^{\infty,1}\) never reach \([1,b_N]\) up to time \(g(\xi^N,c)\),  so that there is no slow-down of particles initially in $[1;b_N]$ inside $[1;b_N]$, yielding the result.  

\end{proof}

Next we come to $\eta^{\infty,2}$ (see again Figure 1). 

\begin{prop}\label{prop:eta1_eta2}
Define a configuration  $ \eta^{\infty,2}\in \{0,1\}^{\Z_{\geq 1}}$  via 
\[
  \eta^{\infty,2}(i) :=
  \begin{cases}
    0, & i \in [1, N^{\delta}],\\[3pt]
    \eta^\infty(i - N^{\delta}), & i > N^{\delta}.
  \end{cases}
\]
Then under the basic coupling
\[
  \mathcal{L}(\eta^{\infty,2}_{g(\xi^N,c)})\ \ge\ \mathcal{L}(\eta^{\infty,1}_{g(\xi^N,c)}).
\]
\end{prop}

\begin{proof}
This follows from $\eta^{\infty,2} \succeq \eta^{\infty,1}$. 
\end{proof}

Finally, we extend $\eta^{\infty,2}$ to $\Z$ by inserting holes everywhere to the left of $1$. 
\begin{prop}\label{prop:eta2_eta3}
Define an extension of \(\eta^{\infty,2}\) to $\Z$  by
\[
  \eta^{\infty,3}(i) :=
  \begin{cases}
    0, & i \leq  0,\\[3pt]
    \eta^{\infty,2}(i), & i \ge 1.
  \end{cases}
\]
Then
\[
  \lim_{N\to\infty}\mathbb{P}\!\left(\,
    \mathcal{L}(\eta^{\infty,3}_{\,g(\xi^{N},c)})
    \ \ge\
    \mathcal{L}(\eta^{\infty,2}_{\,g(\xi^{N},c)})\right) = 1.
\]
\end{prop}

\begin{proof}
The only way we could have $\mathcal{L}(\eta^{\infty,3}_{\,g(\xi^{N},c)})
    <
    \mathcal{L}(\eta^{\infty,2}_{\,g(\xi^{N},c)})$ is if $\mathcal{L}(\eta^{\infty,3})$ jumps from $1$ to $0$ at some time point $ s\in [0,\g]$. 
    Since  for all $t$ $\LL( \eta^{-\mathrm{step}(N^{\delta})}_{t}) \leq \LL(\eta^{\infty,3}_{t})$ we see  that the probability of $ \LL(\eta^{\infty,3})$ reaching position $1$ converges to $0$ by Proposition \ref{block}, finishing the proof. 

\end{proof}
Finally, we reinsert all the particles from $\eta^{\infty}$ that were initially to the left of $1$.
\begin{prop}\label{prop:eta3_eta4}
Define the right-shifted configuration
\[
  \eta^{\infty,4}(i) := \eta^\infty(i - N^{\delta}), \qquad i\in\mathbb{Z}.
\]
Set
\(
  \hat{x}^{*}
  := \inf\{\,i> N^{\delta}:\ \eta^{\infty,4}(i)=1\,\},
\)
and let \(\hat{x}^{*}_{t}\) denote the position at time \(t\) of the particle initially at \(\hat{x}^{*}\).
Then under the basic coupling
\[
  \hat{x}^{*}_{\,g(\xi^{N},c)}
  \ \geq\
  \mathcal{L}\big(\eta^{\infty,3}_{\,g(\xi^{N},c)}\big).
\]
\end{prop}

\begin{proof}
Note that  $ \hat{x}^{*}_{0}
=
  \mathcal{L}\big(\eta^{\infty,3}_{0})$ and that the difference between $\hat{x}^{*}$ and $  \mathcal{L}\big(\eta^{\infty,3}\big)$ is that 
  $\hat{x}^{*}$ may have additional particles to its left initially. Since this can never slow down $ \hat{x}^{*}$, the claim follows. 

\end{proof}

\subsection{Proof of Theorem~\ref{lower_bound_thm1}}

By successive application of
Propositions~\ref{prop:eta1_xi_N}–\ref{prop:eta3_eta4}, we obtain
\[
 \lim_{N\to\infty}\Pb(   \mathcal{L}\!\big(\xi^N_{g(\xi^N,c)}\big)
  \;\le\;
\hat{x}^{*}_{g(\xi^N,c)})=1.
\]
Since \(\eta^{\infty,4}\) is a spatial shift of \(\eta^\infty\) by \(N^{\delta}\) we have 
\[
  \hat{x}^{*}_{g(\xi^N,c)}
  \;=\;
  X_{k_N}\!\big(g(\xi^N,c)\big) + N^{\delta}.
\]
Therefore,
\[
  \lim_{N\to\infty}
  \mathbb{P}\!\left(
    \mathcal{L}\!\big(\xi^N_{g(\xi^N,c)}\big)
    \le 
    X_{k_N}\!\big(g(\xi^N,c)\big) + N^{\delta}
  \right) = 1,
\]
which completes the proof. \qed

\section{Bounds on overtaking particles}\label{secovertake}
Here we provide bounds on the number of holes resp. second class particles which overtake a second class particle resp. first class particle initially to  their right. 
For a single second class particle, Theorem 1.6 of \cite{DH25}  provides precise bounds using a different method. 
We thank Dominik Schmid for explaining and discussing the present method, based on censoring and originating in  \cite{BBHM}, with us.

Intuitivly, it is very natural that a first class particle will not be overtaken by too many second class particles which are initially to its left.
Likewise,  a second class particle should not be overtaken by too many holes which are initially to its left.  
The base case, where a coupling argument suffices to show this, is  reversed step initial data: 
Recall we defined for $Z\in \Z$  the ASEP  $ ( \eta^{-\mathrm{step}(Z)}_{\ell}, \ell\geq 0) $ started from the reversed step initial data
\begin{equation*}
\eta^{-\mathrm{step}(Z)}(j)=\mathbf{1}_{\Z_{\geq Z}}(j),
\end{equation*}
for $Z=0$ we simply write $\eta^{-\mathrm{step}}$.

For reversed step initial data, we have the following bound  from \cite{N20AAP}.
\begin{prop}[Proposition 3.1 in \cite{N20AAP}]\label{block}
Consider  ASEP with reversed step  initial data  $\eta^{-\mathrm{step}(Z)} , $ let $\delta>0$ and denote by $\mathcal{L}(\eta^{-\mathrm{step}(Z)}_{s})$ the position of the leftmost particle of $\eta^{-\mathrm{step}(Z)}_{s}$. Then  there is a $t_0$ such that for $t>t_{0},R\in \Z_{\geq 1}$ and constants $C_{1},C_{2}>0$ (which depend on $p$) we have
\begin{align}\label{bms2}
&\Pb\left(  \mathcal{L}(\eta^{-\mathrm{step}(Z)}_{t})<Z-R\right)\leq C_{1}e^{-C_{2}R}
\\&\label{bms}
\Pb\left( \inf_{0\leq \ell \leq t} \mathcal{L}(\eta^{-\mathrm{step}(Z)}_{\ell })<Z-t^{\delta}\right)\leq C_{1}e^{-C_{2}t^{\delta}}.
\end{align}
\end{prop}
The following is our first Theorem bounding the number of holes which overtake a leftmost second class particle. 
\begin{tthm}\label{61}
Let $\eta \in\{0,1,2\}^{\Z}$, and for $i=0,1,2$  let $A_{i}=\{j\in \Z:\eta(j)=i\}.$
We assume that $ A_{0},A_{2}$ are infinite sets such that 
\begin{equation*}
\max A_{0}<\min A_{2}.
\end{equation*}
Let $\mathcal{L}^{2}(\eta_{s})$ denote the position of the leftmost $2$ of $\eta_{s}$.
Let $M_s $ be the number of $0's$ to the right of $\mathcal{L}^{2}(\eta_{s})$ and let $\delta>	0$. Then  there is a $t_0$ such that for $t>t_{0}$ and constants $C_{1},C_{2}>0$  we have
\begin{equation*}
\Pb\left(\sup_{0\leq s\leq t}M_{s}>t^{\delta}\right)\leq C_{1}e^{-C_{2}t^{\delta}}.
\end{equation*}
\end{tthm}
\begin{proof}
For each $\eta_s$, we produce an element $\eta^{*}_s\in \{0,2\}^{\Z}$ as follows. 
In $\eta_s$, we remove all $1's $ from the particle configuration, merging an edge whose right vertex was deleted with the  leftmost  edge  to its right whose left vertex was deleted. 
This leads to a sequence of $0's$ and $2's$ which becomes uniquely determined by imposing that it is an element of 
$$\left\{\eta\in \{0,2\}^{\Z}: \sum_{i=0}^{\infty}(2-\eta(i))=\sum_{-\infty}^{-1}\eta(i)<\infty\right\}.$$
We call $\eta^{*}_s$ this sequence. Note that by construction $\eta^{*}_0=2*\eta^{-step(0)}$. 
 It is a priori unclear why the stochastic process $(\eta^{*}_s, s\geq 0)$ should follow any exclusion dynamics. We however claim that  $(\eta^{*}_s, s\geq 0)$ is an ASEP with censoring. We first note that nothing changes upon replacing $(\eta^{*}_s,s \geq 0)$  by $(\frac{1}{2}*\eta^{*}_s,s \geq 0)$. The claim thus is that $(\frac{1}{2}*\eta^{*}_s,s \geq 0)$ is an ASEP started from $\eta^{-step(0)}$ with censoring.  The censoring scheme $\mathcal{C}$ is as follows: An edge in $(\eta^{*}_s,s \geq 0)$ is censored at a time $t$ if it results from the merge of edges in $\eta_t$. 
 
To prove the claim, we consider an update of $\eta_t$ along an edge $(i,i+1)$. If $\eta_t (i)=\eta_{t}(i+1)$, both  $\eta_t, \eta^{*}_t$ do not change. If $\eta_t (i)=1$ or $\eta_t (i+1)=1$,  updating the edge $(i,i+1)$  will not change the configuration $\eta^{*}_t$
. If $\{\eta_t (i),\eta_t (i+1)\}=\{0,2\}$, then the edge $(i,i+1)$ does not merge and  
a swap  in $\eta_t$ will likewise be performed 
in $\eta^{*}_t$.  Since these are all possible updates in $\eta_t$, we see that the $0's$ and $2's$ of $(\eta^{*}_s,s\geq 0)$ follow an exlusion dynamics  where an edge $(j,j+1)$ is censored in $\eta^{*}_t$  if it results from a merger, because then the $0,2$ occupying $j,j+1$ in $\eta^{*}_t$ are not nearest neighbours in $\eta_t$ - hence the $0,2$ cannot swap their positions in $\eta_t$, and thusly not in $\eta^{*}_t$. 


  
Then, by Proposition \ref{censor} and Proposition   \ref{block} we have
\begin{align*}
\Pb\left(\inf_{0\leq s\leq t}\LL^{2}\left(\eta^{*}_s\right)<-t^{\delta}\right)&=\Pb\left(\inf_{0\leq s\leq t}\LL\left(\frac{1}{2}\eta^{*}_s\right)<-t^{\delta}\right)\\&\leq\Pb\left(\inf_{0\leq s\leq t}\LL\left(\eta^{-step(0)}_s\right)<-t^{\delta}\right)
\\&\leq C_1 e^{-C_2 t^{\delta}}.
\end{align*}
Let $M^{*}_s$ be the number of $0's$ to the right of $\LL^{2}\left(\eta^{*}_s\right)$. By construction, we have  $M^{*}_s=M_s$  and  $M^{*}_s=-\LL^{2}\left(\eta^{*}_s\right)$.
Thus we conclude 
\begin{align*}
\Pb\left(\sup_{0\leq s\leq t}M_s >t^{\delta}\right)&=\Pb\left(\sup_{0\leq s\leq t}M^{*}_s >t^{\delta}\right)
\\&=\Pb\left(\inf_{0\leq s\leq t}\LL\left(\eta^{*}_s\right)<-t^{\delta}\right)
\\&\leq C_1 e^{-C_2 t^{\delta}}.
\end{align*}
\end{proof}
We also  obtain the following dual version of Theorem \ref{61}.
\begin{tthm}\label{62}
Let $\eta' \in\{0,1,2\}^{\Z}$, and for $i=0,1,2$  let $A_{i}=\{j\in \Z:\eta'(j)=i\}$
We assume that $ A_{1},A_{2}$ are infinite sets such that 
\begin{equation*}
\max A_{2}<\min A_{1}.
\end{equation*}
Let $\mathcal{L}(\eta'_{s})$ denote the position of the leftmost $1$ of $\eta'_{s}$.
Let $M'_{s} $ be the number of $2's$ to the right of $\mathcal{L}(\eta'_{s})$ and let $\delta>	0$. Then  there is a $t_0$ such that for $t>t_{0}$ and constants $C_{1},C_{2}>0$  we have
\begin{equation*}
\Pb\left(\sup_{0\leq s\leq t}M'_{s}>t^{\delta}\right)\leq C_{1}e^{-C_{2}t^{\delta}}.
\end{equation*}
\end{tthm}
\begin{proof}
The proof is very similar to the proof of Theorem \ref{61}, except that here we construct the sequence $\eta'^{*}$ by deleting all $0's$ from  $\eta'$.   This results in a stochastic process $(\eta'^{*}_{t},t\geq 0)$ which starts from 
$$ 2\mathbf{1}_{\Z_{<0}}+\mathbf{1}_{\Z_{\geq 0}}.$$
Identifying the $2'$ as $0's$, we see as in the proof of Theorem   \ref{61} that $(\eta'^{*}_{t},t\geq 0)$ can be identified as an ASEP started from $\eta^{-step(0)}$ with censoring. The position of the leftmost particle at time $s$  of this ASEP with censoring then equals $-M^{'}_s$, finishing the proof. 

\end{proof}
We actually will use the following variant of Theorem \ref{62}.
\begin{cor}\label{cor62}
Let $(\eta^{'N},N\geq 1)$ be a sequence of initial data, and let $A_{i}(N)=\{j\in \Z: \eta^{'N}(j)=i\}$. Assume that 
$$\max A_{2}(N)<\min A_{1}(N), \quad N\geq 1.$$ Let $\delta,D>0$ and denote by  $M'_{s}(N) $ the number of $2's$ to the right of $\mathcal{L}(\eta^{'N}_{s})$. 
Then there is a $N_0$ such that for $N>N_0$ and  constants $\tilde{C}_1,\tilde{C}_2 $ 
we have 
\begin{equation}\label{NNN}
\Pb\left(\sup_{0\leq s\leq D N}M'_{s}(N)>N^{\delta}\right)\leq \tilde{C}_{1}e^{-\tilde{C}_{2}N^{\delta}}.
\end{equation}
\end{cor}
\begin{proof}
We first note that the constants in Theorem \ref{61}  - and thus in Theorem \ref{62} - do not depend on the configuration $\eta$ resp. $\eta'$. They are the constants from Proposition \ref{block}. 
Hence \eqref{NNN} holds by Theorem \ref{62} for every $N$  when $A_1 (N), A_2(N)$ are infinite. When $A_2 (N)$ is finite and $A_{1}(N)$ is infinite, we may turn all $0's$ to the left of $\min A_2 (N)$ into $2's$.
This cannot decrease $M'_{s}(N)$ , and the statement follows from Theorem \ref{62}. 

Finally, if $A_1 (N)$ is finite,  we may add an infinite group of $1's$ so far to the right that it does not affect the evolution of the particles of 
$(\eta^{'N},N\geq 1)$  during $[0,DN]$. Specifically, take $N_0$ large and consider for $N\geq N_0$  the initial configuration 
$$\hat{\eta}^{'N}=\eta^{'N}+\mathbf{1}_{[\max A_{1}(N)+N^{N};\infty)}.$$
The probability that a particle initially present in $A_1(N), A_2(N)$ is affected during $[0,DN]$ by the particles from  $[\max A_{1}(N)+N^{N};\infty)$ can be bounded by the probability that a rate $pDN$ Poisson variable exceeds the value $N^{N}/2$, since for this to happen, at least one  particle from $A_1 (N), A_2 (N),[\max A_{1}(N)+N^{N};\infty)$  would need to travel a distance  $N^{N}/2$ during $[0,DN]$ . Thus this probability can be crudely  bounded by $C_3 e^{-C_4 N^{N/2}}.$
In particular, we have $$\Pb\left(\sup_{0\leq s\leq D N}M'_{s}(N)\neq \sup_{0\leq s\leq D N}\hat{M}'_{s}(N)\right)\leq C_3 e^{-C_4 N^{N/2}}.$$
Hence we may apply Theorem \ref{62}  to  $ \hat{\eta}^{'N}$ and obtain 
\begin{align}
\Pb\left(\sup_{0\leq s\leq D N}M'_{s}(N)>N^{\delta}\right)\leq C_{1}e^{-C_{2}N^{\delta}}+ C_3 e^{-C_4 N^{N/2}}\leq 2C_{1}e^{-C_{2}N^{\delta}},
\end{align}
finishing the proof with $\tilde{C}_1=2C_1, \tilde{C}_2=C_2$.
\end{proof}
\section{Upper bound}\label{secupper}
The aim of this section is to establish $1-F$ as upper bound to the TV distance. For this, we upper bound the asymptotics of the TV distance by  the asymptotics of a Hitting time $\mathfrak{H}$. The main task then is to show that the probability that $\mathfrak{H}$ exceeds $\g$ converges to $1-F(c),$ which is what we do in Theorem \ref{HF}. 
\subsection{Hitting time }
We define the configuration 
\begin{equation}\label{ximax}
\{0,1\}^{[1;b_N]}\ni \xi^{\max}=\mathbf{1}_{[b_N-k_N+1;b_N]}.
\end{equation}
We consider the extension $\bar{\xi}^{\infty}$ of $\xi^{N}$ to $\Z$ obtained by removing the boundaries of the segment, i.e.
\begin{equation}\label{ext1}\bar{\xi}^{\infty}(i)=\begin{cases} \xi^{N}(i),\, i\in [1;b_{N}]\\
0\quad  \,\, \, i\in \Z\setminus[1;b_{N}].
\end{cases}\end{equation}
We define  particle configurations on $\Z$ given by 
\begin{equation}
\label{zeta}
\{0,1\}^{\Z}\ni\zeta^{N}=\bar{\xi}^{\infty}+\mathbf{1}_{\Z_{>b_{N}}}
\end{equation}
\begin{equation}
\label{zeta2}
\{0,1\}^{\Z}\ni \zeta^{\max}=\mathbf{1}_{\Z>(b_N-k_N)}.
\end{equation}
Given initial configurations  $\xi\in \{0,1\}^{[1;b_N]},\zeta\in \{0,1\}^{\Z}$ we define the respective hitting times of $\xi^{\max}, \zeta^{\max}$ as 

\begin{equation}
\mathfrak{g}(\xi)=\inf\{t:\xi_{t}=\xi^{\max}\} \quad \mathfrak{H}(\zeta)=\inf\{t:\zeta_{t}=\zeta^{\max}\}.
\end{equation}
Then, under the basic coupling, we have the inequality 
\begin{equation}\label{20}
\mathfrak{g}(\xi^{N})\leq \mathfrak{H}(\zeta^{N}).
\end{equation}
The inequality \eqref{20} was proven in equation  (20) of \cite{BN22} for the case $b_N=N, \xi^{N}=\mathbf{1}_{[1;k_N]}$ but the  same proof verbatim shows \eqref{20} for arbitrary $\xi^{N}$ and $b_N$. 
\begin{prop}\label{dprop}
We have 
\begin{equation}
\limsup_{N\to\infty}d_{\xi^{N}}(g(\xi^{N},c))\leq \limsup_{N\to\infty} \Pb(\mathfrak{H}(\zeta^{N})>g(\xi^{N},c))
\end{equation}

\end{prop}
\begin{proof}
We start at time $0$ an ASEP from the stationary distribution and call 
 $\pi_{t}^{b_N,k_N}$  the state of  this  ASEP at time $t.$ 
 By Proposition 4.7 of \cite{LPW17}, we have for the basic coupling (in fact, any coupling) the inequality 
 \begin{align}
 d_{\xi^{N}}(t)\leq \Pb(\pi_{t}^{b_N,k_N}\neq \xi^{N}_t).
 \end{align}
Now if $\pi_{t}^{b_N,k_N}\neq \xi^{N}_t$ we cannot have $\max\{\mathfrak{g}(\xi^{N}),\mathfrak{g}(\pi_{0}^{b_N,k_N})\}\leq t$ since in that case the two processes must have coalesced before time $t$. Consequently we have 
 \begin{align}
 d_{\xi^{N}}(t)\leq \Pb(\mathfrak{g}(\pi_{0}^{b_N,k_N})\geq t)+\Pb(\mathfrak{g}(\xi^{N})\geq t).
 \end{align}
Setting  $\bar{\pi}_{0}^{b_N,k_N}(j):=\pi_{0}^{b_N,k_N}(j), j\in [1;b_N],$  and $\bar{\pi}_{0}^{b_N,k_N}(j):=0$ for $j\in \Z\setminus[1;b_N]$, 
we define  a configuration  $\pi_{0}^{b_N,k_N,\infty}$ on $\Z$ via 
$$\{0,1\}^{\Z}\ni \pi_{0}^{b_N,k_N,\infty}(j)= \bar{\pi}_{0}^{b_N,k_N}(j)+\mathbf{1}_{\Z>b_N}(j) $$
and  thus obtain from \eqref{20} that 
 \begin{align}
 d_{\xi^{N}}(t)\leq \Pb(\mathfrak{H}(\pi_{0}^{b_N,k_N,\infty})\geq t)+\Pb(\mathfrak{H}(\xi^{N})\geq t).
 \end{align}
 Hence it remains to show 
  \begin{align}
\lim_{N\to\infty}\Pb(\mathfrak{H}(\pi_{0}^{b_N,k_N,\infty})\geq \g)=0. 
 \end{align}
To see this, define for $l\in \N$ the events 
\begin{align}
S_{b_N}(l)=\{\LL(\pi_{0}^{b_N,k_N,\infty})<b_N-k_N-l\}\quad \hat{S}_{b_N}(l)=\{\RR(\pi_{0}^{b_N,k_N,\infty})>b_N-k_N+l\}.
\end{align}
It follows from Proposition 4.2 of \cite{BN22} and particle-hole duality that there are constants $C_1, C_2$ independent of $N$ such that 
\begin{equation}
\Pb(S_{b_N}(l)\cup \hat{S}_{b_N}(l))\le C_1 e^{-C_{2}l}.
\label{CC}
\end{equation}
Let now $R_{b_N}(l):=\Omega\setminus(S_{b_N}(l)\cup \hat{S}_{b_N}(l))$ and define the configuration on $\Z$ 
\begin{equation}
I^{l}=\mathbf{1}_{[b_{N}-k_{N}-l+1;b_N-k_N]}+\mathbf{1}_{\Z_{>b_N-k_N+l}}.
\end{equation}
We then have $R_{b_N} (L)\subseteq \{I^{l}\preceq  \pi^{b_N,k_N}_{0}\}$ and thus in particular
\begin{equation}\label{23}
R_{b_N} (L)\subseteq \{\mathfrak{H}(I^{l})\geq \mathfrak{H}(\pi_{0}^{b_N,k_N,\infty})\}.
\end{equation}
Now 
Theorem 1.9 of \cite{BBHM}\footnote{See Theorem 5 of \cite{BN22} for a reformulation of Theorem 1.9 of \cite{BBHM} which applies here after a shift of the configuration $I^{M}$ of Theorem 5 of \cite{BN22} .} implies that for $l=\g^{1/10}$ we have 
\begin{equation}
\label{24}
\lim_{N\to\infty}\Pb(\mathfrak{H}(I^{\g^{1/10}})\geq \g)=0. \end{equation}
At the same time, by \eqref{CC} we have $\lim_{N\to\infty}\Pb(R_{b_N} (\g^{1/10})=1$. 
Putting everything together, we thus obtain 
  \begin{align}
&\lim_{N\to\infty}\Pb(\mathfrak{H}(\pi_{0}^{b_N,k_N,\infty})\geq \g)
\\&=\lim_{N\to\infty}\Pb(\{\mathfrak{H}(\pi_{0}^{b_N,k_N,\infty})\geq \g \}\cap R_{b_N} (\g^{1/10}))
\\&\leq \lim_{N\to\infty}\Pb(\mathfrak{H}(I^{\g^{1/10}})\geq \g)=0, \label{25}
 \end{align}
 where the left  inequality in \eqref{25} is \eqref{23} and the equality  in \eqref{25} is \eqref{24}.

\end{proof}
The following Theorem will establish $1-F(c)$ as upper bound to the TV distance. 
Its proof relies on Corollaries \ref{cor} and \ref{cor2}, proven in Section \ref{5}.
\begin{tthm}\label{HF}
Under the assumptions of Theorem \ref{main}, we have 
\begin{equation}
\limsup_{N\to\infty} \Pb(\mathfrak{H}>g(\xi^{N},c))=1-F(c). 
\end{equation}
\end{tthm}
\begin{proof}[ Proof (assuming  Corollaries \ref{cor} and \ref{cor2})]
We define, as in \cite {BN22}, the event
$$B_{N}(c)=\left\{\RR(\zeta^{N}_{\g})\leq b_{N}-k_{N}+N^{\kappa})\right\}\cap \left\{\LL(\zeta^{N}_{\g})\geq b_{N}-k_{N}-N^{\kappa})\right\}$$
and show 
\begin{equation}\label{BNC}\lim_{N\to\infty}\Pb(B_{N}(c))=F(c).\end{equation}
To prove this, let us label the holes of $\zeta^{N}$ from right to left as 
$$\cdots<H^{N}_{b_N-k_{N}-2}<H^{N}_{b_N-k_{N}-1}<H^{N}_{b_N-k_{N}}=\RR(\zeta^{N}_{}).$$
Since under basic coupling we have (recall $H_{b_N-k_N+j}$ denote the holes of the extension $\eta^{\infty}$ as in \eqref{holeslab}) $H_{b_N-k_{N}-N^{\delta}}(\g)\leq H^{N}_{b_N-k_{N}-N^{\delta}}(\g)$ it follows 
$$H_{b_N-k_{N}-N^{\delta}}(\g)\leq H^{N}_{b_N-k_{N}-N^{\delta}}(\g)\leq \RR(\zeta^{N}_{\g})$$
implying for any $\kappa, \delta\in (0,1/3)$  by Corollary \ref{cor2} and \eqref{holes}
\begin{equation}\label{HN}
\lim_{N\to\infty}\Pb(H^{N}_{b_N-k_{N}-N^{\delta}}(\g)\leq b_{N}-k_{N}+N^{\kappa})=F(c). 
\end{equation}
The crucial point now is that because of \eqref{HN}, it is very unlikely that $H^{N}_{b_N-k_{N}-N^{\delta}}$ has reached position $ b_{N}-k_{N}+N^{\kappa}$ at time $\g$  but $H^{N}_{b_N-k_{N}}$ has not: By Corollary \ref{cor2} and \eqref{HN} we have 
\begin{equation}
\begin{aligned}\label{lim0}
\lim_{N\to\infty}&\Pb(\{H^{N}_{b_N-k_{N}-N^{\delta}}(\g)\leq b_{N}-k_{N}+N^{\kappa}\}\setminus \{H^{N}_{b_N-k_{N}}(\g)\leq b_{N}-k_{N}+N^{\kappa}\})\\&=0.
\end{aligned}
\end{equation}

Since to the left of the leftmost particle everything is filled with holes, we have the identity of events
\begin{align*}
\left\{\LL(\zeta^{N}_{\g})> b_{N}-k_{N}-N^{\kappa})\right\}=\left\{H_{b_N-k_N-N^{\kappa}}^{N}(\g)=b_N-k_N-N^{\kappa}\right\}.
\end{align*}
Hence we may likewise write 
\begin{equation}
\begin{aligned}\label{intersect}
\left\{\LL(\zeta^{N}_{\g})> b_{N}-k_{N}-N^{\kappa})\right\}=&\left\{\LL(\zeta^{N}_{\g})> b_{N}-k_{N}-N^{\kappa})\right\}\\&\cap \left\{H^{N}_{b_N-k_N-N^{\kappa}}(\g)\leq b_N+k_N+N^{\kappa}\right\} .
\end{aligned}
\end{equation} It follows directly from  Corollary \ref{cor} that 
\begin{align*}
F(c)&=\lim_{N\to\infty}\Pb\left(\left\{\LL(\zeta^{N}_{\g})> b_{N}-k_{N}-N^{\kappa})\right\}\right)=\lim_{N\to\infty}\Pb\left(\left\{\LL(\zeta^{N}_{\g})\geq  b_{N}-k_{N}-N^{\kappa})\right\}\right).
\end{align*}
Applying \eqref{lim0}  to \eqref{intersect} yields 
\begin{align*}
F(c)&=\lim_{N\to\infty}\Pb\left(\left\{\LL(\zeta^{N}_{\g})>  b_{N}-k_{N}-N^{\kappa})\right\}\right)
\\&=\lim_{N\to\infty}\Pb\left(\left\{\LL(\zeta^{N}_{\g})> b_{N}-k_{N}-N^{\kappa})\right\}\cap \left\{H^{N}_{b_N-k_N-N^{\kappa}}(\g)\leq b_N+k_N+N^{\kappa}\right\} \right)\\&=\lim_{N\to\infty}\Pb\left(\left\{\LL(\zeta^{N}_{\g})\geq b_{N}-k_{N}-N^{\kappa})\right\}\cap \left\{H^{N}_{b_N-k_N}(\g)\leq b_N+k_N+N^{\kappa}\right\} \right),
\end{align*}
showing \eqref{BNC}.
From now on, we can continue 
as in  \cite{BN22}, in particular we can show as in Proposition 3.2 of \cite{BN22} that for $\hat{\kappa}\in (\kappa,1/3)$ 
we have\footnote{In \cite{BN22}, $\hat{\kappa}=1/5$, but this plays no role.} 
$$\limsup_{N\to\infty} \Pb(\{\mathfrak{H}>g(\xi^{N},c)+N^{\hat{\kappa}}\}\cap B_{N}(c))=0. $$
Thus,  for any $\varepsilon>0$ and $\hat{\kappa}\in (\kappa,1/3),$   we obtain 
\begin{align*}
\limsup_{N\to\infty}\Pb(\mathfrak{H}>g(\xi^{N},c))&\leq \limsup_{N\to\infty} \Pb(\mathfrak{H}>g(\xi^{N},c-\varepsilon)+N^{\hat{\kappa}})\\&\leq \lim_{N\to\infty} \Pb(\{\mathfrak{H}>g(\xi^{N},c-\varepsilon)+N^{\hat{\kappa}}\}\cap B_{N}(c-\varepsilon))\\&+  \lim_{N\to\infty} 1-\Pb(B_{N}(c-\varepsilon))
\\&=1-F(c-\varepsilon).
\end{align*}
Since $\varepsilon$ is arbitrary and $F$ is continuous, the statement follows. 
\end{proof}
We report as Corollary the obtained upper bound on the TV-distance: 
\begin{cor}\label{corfinal}
Under the assumptions of Theorem \ref{main}, we have 
\begin{equation}
\limsup_{N\to\infty}d_{\xi^{N}}(g(\xi^{N},c))\leq 1-F(c).
\end{equation}
\end{cor}
\begin{proof}
This directly follows from Theorem \ref{HF} and Proposition \ref{dprop}.
\end{proof}
\subsection{Proof of limit Theorem for the hitting time}\label{5}
We start by showing that  the convergence \eqref{particles}, valid for $X_{k_N},$ still holds upon deleting all particles $\eta^{\infty}$  may have to the left of $1$. 
\begin{tthm}\label{theorem5}
 Let $\eta^{\infty}$ be the extension from Theorem \ref{main}. Let $\hat{\xi}^{\infty}=\eta^{\infty}*\mathbf{1}_{\Z_{\geq 1}}$ be the configuration obtained from $\eta^{\infty}$ by replacing all particles on non-positive integers by holes. 
 Then, for any $\kappa\in (0,1/3)$, 
 \begin{equation}\label{F(c)}
 \lim_{N\to \infty}\Pb\left(  \mathcal{L}\left(\hat{\xi}^{\infty}_{g(\xi^{N},c)}\right)\leq b_{N}-k_{N} -N^{\kappa}\right)=1-F(c).
 \end{equation}
\end{tthm}
\begin{proof}
We couple the ASEPs started from $\eta^{\infty}, \hat{\xi}^{\infty}$ together via basic coupling. This induces  an  ASEP with first class particles initially located on $\{j\geq 1:  \hat{\xi}^{\infty}(j)=1\}$ and second class particles initially located on $\{j\leq 0: \eta^{\infty}(j)=1\}$, i.e. an ASEP with initial configuration 
\begin{equation}\label{2ndclass}\mathbf{1}_{\{j\geq 1:  \hat{\xi}^{\infty}(j)=1\}}+2*\mathbf{1}_{\{j\leq 0: \eta^{\infty}(j)=1\}}.\end{equation}
See Figure 2 for an illustration. 
We label the 2's of  \eqref{2ndclass} from right to left as 
\begin{equation}
\cdots<Y_3<Y_2<Y_1 \leq 0
\end{equation}
and denote by $Y_{k}(s)$ the positon  at time $s$ of the second class particle initially at position $Y_k$. 
Let $\LL(\hat{\xi}^{\infty}_{s})$ be the position of the leftmost particle in  $\hat{\xi}^{\infty}_{s}$ .
It follows from Corollary \ref{cor62} that for $\delta>0,$ with probability going to $1$ as $N\to \infty$    there were never more than $(g(\xi^{N},c))^{\delta}$ many  2's   to the right of  $\LL(\hat{\xi}^{\infty}_{\cdot })$ during the time  interval  $[0,g(\xi^{N},c)]$. Since $\g$ is of order $N$, taking $\delta\in (0,1/3), $ and a $\delta' \in (\delta,1/3)$, we see that, as $N\to \infty$,  the number of 2's to the right of  $\LL(\hat{\xi}^{\infty}_{\cdot })$ never exceeded $N^{\delta'}$ with probability going to $1$.  This is equivalent to the statement
\begin{equation*}\label{E1N}
\lim_{N\to\infty}\Pb\left(E_{1,N}\right)=1,
\end{equation*}
where  we define $E_{1,N}:=\{\max_{0\leq s \leq g(\xi^{N},c)}(Y_{N^{\delta'}}(s)-\LL(\hat{\xi}^{\infty}_{s }))<0\}.$  See Figure 2.

On the event $E_{1,N}$, $Y_{N^{\delta'}}$ never swapped its position with a first class particle during $[0,g(\xi^{N},c)]$. Denote by $X_{k_{N}+N^{\delta'}}(g(\xi^{N},c))$ the position at time $g(\xi^{N},c)$ of the first class particle in the $(\eta^{\infty}_s, s\geq 0)$ process which started 
at position $Y_{N^{\delta'}}(0)$. Since, on $E_{1,N}$, $Y_{N^{\delta'}}(\cdot )$ never swapped its position with a first class particle during $[0,g(\xi^{N},c)]$, it follows that,  on $E_{1,N}$, $Y_{N^{\delta'}}(\cdot )$ and  $X_{k_{N}+N^{\delta'}}(\cdot)$ agree during  $[0,g(\xi^{N},c)]$. 
Hence 
\begin{equation}
\begin{aligned}\label{forpatrik}
E_{1,N}&=E_{1,N}\cap\{X_{k_{N}+N^{\delta'}}(g(\xi^{N},c))=Y_{N^{\delta'}}(g(\xi^{N},c))\}
\\&\subseteq  \{\max_{0\leq s \leq g(\xi^{N},c)}(X_{k_N+N^{\delta'}}(s)-\LL(\hat{\xi}^{\infty}_{s }))<0\}\\&=:E_{2,N}.
\end{aligned}
\end{equation}
Hence $\lim_{N\to\infty}\Pb\left(E_{2,N}\right)=1$ and we thus obtain using assumption \eqref{particles}
\begin{align*}
1-F(c)&=\lim_{N\to\infty}\Pb\left( X_{k_{N}+N^{\delta'}}(g(\xi^{N},c))  \leq b_{N}-k_{N} -N^{\kappa}\right)
\\&=\lim_{N\to\infty}\Pb\left(E_{2,N}\cap\{ X_{k_{N}+N^{\delta'}}(g(\xi^{N},c))  \leq b_{N}-k_{N} -N^{\kappa}\}\right)
\\&\geq \lim_{N\to\infty}\Pb\left(E_{2,N}\cap\{ \mathcal{L}\left(\hat{\xi}^{\infty}_{g(\xi^{N},c)}\right) \leq b_{N}-k_{N} -N^{\kappa}\}\right)
\\&= \lim_{N\to\infty}\Pb\left(\mathcal{L}\left(\hat{\xi}^{\infty}_{g(\xi^{N},c)}\right) \leq b_{N}-k_{N} -N^{\kappa}\right).
\end{align*}
On the other hand, under the basic coupling, we always have $\mathcal{L}\left(\hat{\xi}^{\infty}_{g(\xi^{N},c)}\right) \leq X_{k_{N}}(g(\xi^{N},c)) $ because $X_{k_{N}}$ may have  additional particles to its left which could  block its jumps to the left, whereas $\mathcal{L}\left(\hat{\xi}^{\infty}_{g(\xi^{N},c)}\right) $ can always jump to the left. Hence we also have  using assumption \eqref{particles}
\begin{align*}
1-F(c)&=\lim_{N\to\infty}\Pb\left( X_{k_{N}}(g(\xi^{N},c))  \leq b_{N}-k_{N} -N^{\kappa}\right)
\\&\leq  \lim_{N\to\infty}\Pb\left(\mathcal{L}\left(\hat{\xi}^{\infty}_{g(\xi^{N},c)}\right) \leq b_{N}-k_{N} -N^{\kappa}\right),
\end{align*}
finishing the proof.
\end{proof}
 
\begin{figure}
\begin{center}
\begin{tikzpicture}[scale=1.6]

\def\xL{-3.2}
\def\xR{ 6.0}

\def\sites{-3.0,-2.6,-2.2,-1.8,-1.4,-1.0,-0.6,-0.2,%
            0.2,0.6,1.0,1.4,1.8,2.2,2.6,3.0,3.4,3.8,4.2,4.6,5.0,5.4}

\def\xone{0.2}   
\def\xbN{3.8}    

\newcommand{\hole}[1]{\draw (#1,0.2) circle (0.1);}
\newcommand{\onep}[1]{\filldraw[black] (#1,0.2) circle (0.1);}
\newcommand{\two}[1]{\filldraw[gray] (#1,0.2) circle (0.1);}

\draw[ thick,<->] (\xL,0) -- (\xR,0);

\draw[very thick] (\xone,0.08)--(\xone,-0.08);
\draw[very thick] (\xbN,0.08)--(\xbN,-0.08);
\draw[very thick] (0.6,0.08)--(0.6,-0.08);
\node[below] at (\xone,-0.18) {$1$};
\node[below] at (\xbN,-0.18) {$b_N$};
\node[below] at (0.6+0.1,-0.18) {$\LL(\hat\xi^\infty_s)$};

\foreach \x in \sites {\hole{\x}}

\foreach \x in {-2.6,-2.2,-1.8,-1.0,-0.6}
  {\two{\x}}

\foreach \x in {0.6,1.0,1.4,1.8,2.6,3.0,3.8,5.0, 5.4}
  {\onep{\x}}

\def\xY{-2.2}
\two{\xY}
\draw[thick] (\xY,0.2) circle (0.14);
\node[above] at (\xY,0.38) {$Y_{N^{\delta'}}$};

\begin{scope}[yshift=-1.6cm]
\draw[thick,<->, draw=black] (\xL,0) -- (\xR,0);

\draw[very thick] (\xone,0.08)--(\xone,-0.08);
\draw[very thick] (\xbN,0.08)--(\xbN,-0.08);
\draw[very thick] (-0.6,0.08)--(-0.6,-0.08);

\node[below] at (\xone,-0.18) {$1$};
\node[below] at (\xbN,-0.18) {$b_N$};
\node[below] at (-0.6,-0.18) {$\LL(\hat\xi^\infty_s)$};

\foreach \x in \sites {\hole{\x}}

\foreach \x in {-0.6,-0.2,1.4,1.8,3.0,3.8,4.6,5.0, 5.4}
  {\onep{\x}}

\foreach \x in {-2.6,-1.0,-1.4,0.2, 0.6}
  {\two{\x}}

\def\xY{-1.4}
\two{\xY}
\draw[thick] (\xY,0.2) circle (0.14);
\node[above] at (\xY,0.38) {$Y_{N^{\delta'}}$};

\end{scope}

\end{tikzpicture}
\end{center}

\caption{%
Top row: the initial discrepancy configuration
$\mathbf{1}_{\{j\ge1:\hat{\xi}^\infty(j)=1\}} + 2\,\mathbf{1}_{\{j\le0:\eta^\infty(j)=1\}}$ used in the proof of Theorem \ref{theorem5}. \\
Bottom row: depiction of the event $E_{1,N}$ from \eqref{E1N}: the second-class particle
$Y_{N^{\delta'}}(s)$ remains strictly to the left of the leftmost first-class particle
$\LL(\hat{\xi}^\infty_s)$ for all $s\le g(\xi^N,c)$, and therefore does not swap with
first-class particles.\\
Black circles denote first-class particles, gray circles denote second-class particles, and empty circles denote holes.
}
\label{fig:secondclass_coupling_E1N}
\end{figure}
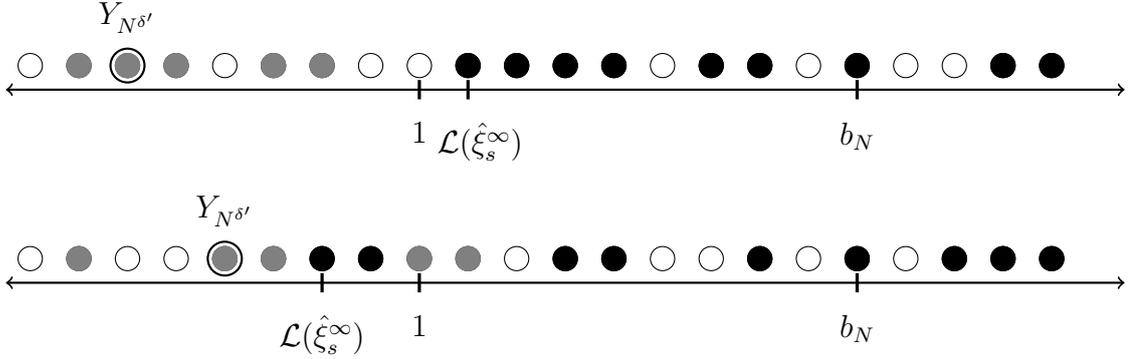
The next goal is to extend the convergence  established in Theorem \ref{theorem5} from $\mathcal{L}\left(\hat{\xi}^{\infty}_{g(\xi^{N},c)}\right)$ to $\LL(\zeta^{N}_{\g})$. 
Recall the configurations $\bar{\xi}^{\infty},\zeta^{N}$ from \eqref{ext1},\eqref{zeta}. Note that among the extensions of $\xi^{N}$ which have no particles in $(-\infty;0]$ the extension  $\bar{\xi}^{\infty}$ is the one where particles move fastest, whereas 
$\zeta^{N}$ is the one where particles move slowest. 
Since  the next theorem shows that it is 
 very unlikely that $\LL(\bar{\xi}^{\infty}_{\g})$ has reached position  $ b_{N}-k_{N}-N^{\kappa}, $ but  $\LL(\zeta^{N}_{\g})$ has not,  and since 
  $\bar{\xi}^{\infty}$ is the fastest extension, the next theorem will apply with $\LL(\bar{\xi}^{\infty})$ replaced by $\mathcal{L}\left(\hat{\xi}^{\infty}_{g(\xi^{N},c)}\right).$
\begin{tthm}\label{theorem6}
Let $\bar{\xi}^{\infty},\zeta^{N}$ be as in \eqref{ext1},\eqref{zeta}. Then, for $\kappa\in (0,1/3)$  we have 
\begin{equation}
\lim_{N\to\infty}\Pb(\{\LL(\bar{\xi}^{\infty}_{\g})\geq b_{N}-k_{N}-N^{\kappa}\}\setminus\{\LL(\zeta^{N}_{\g})\geq b_{N}-k_{N}-N^{\kappa}\})=0. 
\end{equation}
\end{tthm}
\begin{proof}
We couple two ASEPs started from $\bar{\xi}^{\infty},\zeta^{N}$ together via basic coupling. This induces an ASEP with second class particles and initial configuration given by (see Figure 3)
\begin{equation}
\zeta^{2,N}=\bar{\xi}^{\infty}(i)+2*\mathbf{1}_{\Z_{>b_{N}}}(i), \quad i\in \Z.
\end{equation}
Let $\LL^{2}(\zeta^{2,N}_{t})$ be the position of the leftmost $2$ of $\zeta^{2,N}$. It follows from Theorem \ref{61} that, for $\delta>0$,  during  $[0,\g]$, this leftmost $2$ will never have more than $\g^{\delta}$ many $0's$ to its right with probability going to $1$ as $N\to \infty$. More precisely, with $M_s$ as in Theorem \ref{61}, we have 
\begin{equation*}
\Pb\left(\sup_{0\leq s\leq \g}M_{s}>\g^{\delta}\right)\leq C_{1}e^{-C_{2}\g^{\delta}}.
\end{equation*}
 We may take $\delta\in (0,1/3)$, and replace $\g^{\delta}$ by $N^{\kappa}/2$ for a $\kappa\in (\delta,1/3)$ and thus obtain 
 \begin{equation*}
\lim_{N\to \infty}\Pb\left(\sup_{0\leq s\leq \g}M_{s}>N^{\kappa}/2\right)=0. 
\end{equation*}
Let $E_{3,N}:=\{\sup_{0\leq s\leq \g}M_{s}\leq N^{\kappa}/2\}.$ 
 Since $\LL^{2}(\zeta^{2,N})$ can  have  at most $k_{N}$ many  1's to its right, 
 it follows that \begin{equation*}E_{3,N}\subseteq \left\{\inf_{0\leq s\leq \g}\LL^{2}(\zeta^{2,N}_{s})>b_{N}-k_{N}-N^{\kappa}/2\right\}=: E_{4,N}
 \end{equation*}
 and hence $\lim_{N\to \infty}\Pb\left(E_{4,N}\right)=1$.

 Recall now that $\LL(\zeta^{N}),\LL(\bar{\xi}^{\infty})$ start at the same position and are leftmost particles. Hence, for  $\LL(\zeta^{N}_{\g})<\LL(\bar{\xi}^{\infty}_{\g})$ to happen, at some time point during $[0,\g]$  a jump to the right of the particle $\LL(\zeta^{N})$ was blocked by a particle not present in the process $(\bar{\xi}^{\infty}_s,s\geq 0)$. But this blocking  happens exactly iff  the leftmost 1 $\LL(\zeta^{2,N})$ swaps its position with the leftmost 2 $\LL^{2}(\zeta^{2,N})$ during $[0,\g]$. 
 On $E_{4,N}$ this swap can only happen to the right of $b_{N}-k_{N}-N^{\kappa}/2$. See Figure 3. Hence we obtain 
 \begin{equation}
 \begin{aligned}\label{8}
 E_{4,N} \cap \{\LL(\zeta^{N}_{\g})<\LL(\bar{\xi}^{\infty}_{\g})\} &\subseteq \left\{\sup_{0\leq s\leq \g}\LL(\zeta^{N}_{s})\geq b_{N}-k_{N}-N^{\kappa}/2\right\}\\&=:E_{5,N}.
 \end{aligned}
 \end{equation}
 Setting 
 \begin{align*} E_{6,N}:&=\{\LL(\bar{\xi}^{\infty}_{\g})\geq b_{N}-k_{N}-N^{\kappa}\}\setminus\{\LL(\zeta^{N}_{\g})\geq b_{N}-k_{N}-N^{\kappa}\}\\&\subseteq \{\LL(\zeta^{N}_{\g})<\LL(\bar{\xi}^{\infty}_{\g})\}\end{align*}
 we thus see  from \eqref{8}
  \begin{equation}\label{9}
 E_{4,N} \cap E_{6,N}\subseteq E_{5,N}\cap \left\{\LL(\zeta^{N}_{\g})\leq b_{N}-k_{N}-N^{\kappa}\right\}.
 \end{equation}
 The point now is that for the r.h.s of \eqref{9} to happen, $\LL(\zeta^{N})$ would need to first reach position $b_{N}-k_{N}-N^{\kappa}/2$, and then travel back to $b_{N}-k_{N}-N^{\kappa},$ which is very unlikely. To make this precise,  we start an ASEP from $\eta^{-step(b_{N}-k_{N}-N^{\kappa}/2)}$ at time $0$. Under the basic coupling, we have
 \begin{equation*}
 E_{5,N}\subseteq\{ \LL(\zeta^{N}_{\g})\geq \LL(\eta^{-step(b_{N}-k_{N}-N^{\kappa}/2)}_{\g})\}
 \end{equation*}
 and thus 
 \begin{equation}\label{10}
 E_{5,N}\cap \left\{\LL(\zeta^{N}_{\g})\leq b_{N}-k_{N}-N^{\kappa}\right\}\subseteq 
 \{\LL(\eta^{-step(b_{N}-k_{N}-N^{\kappa}/2)}_{\g}\leq  b_{N}-k_{N}-N^{\kappa}\}.
 \end{equation}
 We may then apply the inequality \eqref{bms2} to the r.h.s of \eqref{10} to obtain 
 \begin{equation*}
 \lim_{N\to\infty}\Pb\left( E_{5,N}\cap \left\{\LL(\zeta^{N}_{\g})\leq b_{N}-k_{N}-N^{\kappa}\right\}\right)=0.
 \end{equation*}
  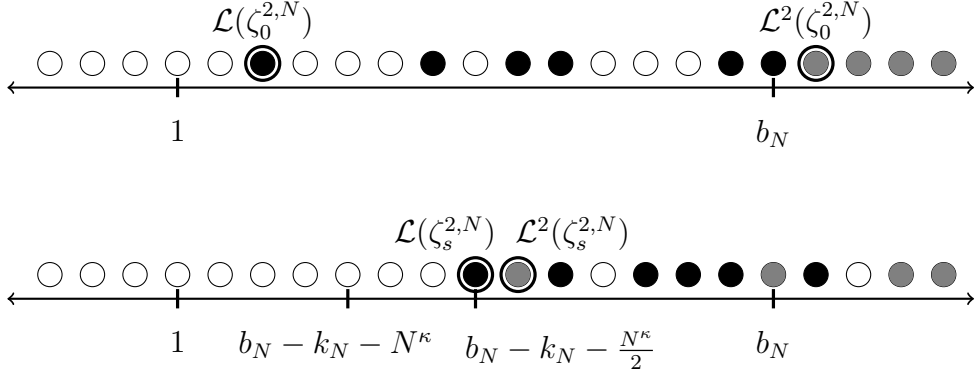
\begin{figure}[H]\begin{center}
\begin{tikzpicture}[scale=1.6]

\def\xL{-1.75}
\def\xR{ 6.20}

\def\leftsites{-1.40,-1.05,-0.70,-0.35}
\def\middlesites{0.00,0.35,0.70,1.05,1.40,1.75,2.10,2.45,2.80,3.15,3.50,3.85,4.20,4.55}
\def\rightsites{4.90,5.25,5.60,5.95}

\def\xone{-0.35}
\def\xbN{4.55}

\def\xA{2.10}
\def\xB{3.85}

\newcommand{\hole}[1]{\draw (#1,0.2) circle (0.1);}
\newcommand{\partone}[1]{\fill (#1,0.2) circle (0.1);}
\newcommand{\parttwo}[1]{\fill[gray] (#1,0.2) circle (0.1);}

\newcommand{\highlight}[1]{\draw[very thick] (#1,0.2) circle (0.14);}
\newcommand{\labelabove}[3]{\node[above=#1] at (#2,0.2) {#3};}

\def\xipartsRowOne{0.35,1.75,2.45,2.80,4.20,4.55}
\def\righttwosRowOne{4.90,5.25,5.60,5.95}

\def\xipartsRowTwo{2.10,2.80,3.50,3.85,4.20,4.90}
\def\righttwosRowTwo{2.45,4.55,5.60,5.95}

\def\LoneRowOne{0.35}
\def\LtwoRowOne{4.90}

\def\LoneRowTwo{2.10}
\def\LtwoRowTwo{2.45}

\newcommand{\DrawBaseRow}{%
  \draw[thick,<->] (\xL,0) -- (\xR,0);
  \draw[very thick] (\xone,0.08)--(\xone,-0.08);
  \draw[very thick] (\xbN,0.08)--(\xbN,-0.08);
  \node[below] at (\xone,-0.18) {$1$};
  \node[below] at (\xbN,-0.18) {$b_N$};

  \foreach \x in \leftsites   {\hole{\x}}
  \foreach \x in \middlesites {\hole{\x}}
  \foreach \x in \rightsites  {\hole{\x}}
}

\DrawBaseRow
\foreach \x in \xipartsRowOne   {\partone{\x}}
\foreach \x in \righttwosRowOne {\parttwo{\x}}

\highlight{\LoneRowOne}
\highlight{\LtwoRowOne}

\labelabove{5pt}{\LoneRowOne}{$\LL(\zeta^{2,N}_{0})$}
\labelabove{5pt}{\LtwoRowOne}{$\LL^{2}(\zeta^{2,N}_{0})$}

\begin{scope}[yshift=-1.75cm]
\DrawBaseRow

\def\xMone{2.10}
\def\xMtwo{1.05}

\draw[very thick] (\xMone,0.08)--(\xMone,-0.08);
\draw[very thick] (\xMtwo,0.08)--(\xMtwo,-0.08);
\node[below] at (\xMtwo-0.1,-0.18) {$b_N-k_N-N^\kappa$};
\node[below] at (\xMone+0.7,-0.18) {$b_N-k_N-\frac{N^\kappa}{2}$};

\foreach \x in \xipartsRowTwo   {\partone{\x}}
\foreach \x in \righttwosRowTwo {\parttwo{\x}}

\highlight{\LoneRowTwo}
\highlight{\LtwoRowTwo}

\labelabove{5pt}{\LoneRowTwo-0.25}{$\LL(\zeta^{2,N}_{s})$}
\labelabove{5pt}{\LtwoRowTwo+0.45}{$\LL^{2}(\zeta^{2,N}_{s})$}

\end{scope}

\end{tikzpicture}\end{center}

\caption{
Top: the initial configuration $\zeta^{2,N}$: all sites $i<1$ are holes, while sites to the right of $b_N$ are occupied by second-class particles.\\
Bottom: $\zeta^{2,N}_s$ for $s\in [0,g(\xi^{N},c)]$.  The leftmost  $2$ is unlikely to reach position $b_N - k_N - \tfrac12 N^{\kappa}$. For the inequality
$\LL\bigl(\zeta^{N}_{g(\xi,c)}\bigr) < \LL\bigl(\bar{\xi}^{\infty}_{g(\xi,c)}\bigr)$
to hold, $\LL\bigl(\zeta^{N})$ must make a swap with the leftmost 2, which can then only take place to the right of $b_N - k_N - \tfrac12 N^{\kappa}$.
But it is highly unlikely that $\LL\bigl(\zeta^{N})$ travels back to $b_N - k_N -  N^{\kappa}.$ Thus in total it is very unlikely that at the same time $\LL\bigl(\zeta^{N}_{g(\xi,c)}\bigr) < \LL\bigl(\bar{\xi}^{\infty}_{g(\xi,c)}\bigr)$
and  $\LL\bigl(\zeta^{N}_{g(\xi,c)}\bigr) < b_N - k_N -  N^{\kappa}.$
Black circles denote first-class particles, gray circles denote second-class particles, and empty circles denote holes. The leftmost first- and second-class particles are highlighted and labeled in each row.
}
\label{fig:overtake_schematic_22sites_mixed}
\end{figure}

 By the inclusion \eqref{9}, we thus see
   \begin{equation}\label{91}
\lim_{N\to\infty}\Pb( E_{4,N} \cap E_{6,N})=0,
 \end{equation}
 and finally, since $\lim_{N\to\infty}\Pb( E_{4,N})=1$, we obtain from \eqref{91}
   \begin{equation*}
\lim_{N\to\infty}\Pb(  E_{6,N})=0,
 \end{equation*}
 which is the desired statement. 
\end{proof}
\begin{cor}\label{cor}
We have 
\begin{equation*}
\lim_{N\to\infty}\Pb(\LL(\zeta^{N}_{\g})\geq b_{N}-k_{N}-N^{\kappa})=F(c). 
\end{equation*}
\end{cor}
\begin{proof}
Using Theorem \ref{theorem6}, we  obtain 
\begin{align*}\lim_{N\to\infty}\Pb(\LL(\zeta^{N}_{\g})\geq b_{N}-k_{N}-N^{\kappa})&=\lim_{N\to\infty}\Pb(\LL(\zeta^{N}_{\g})\geq b_{N}-k_{N}-N^{\kappa})\\&+\lim_{N\to\infty}\Pb(E_{6,N})
\\&= \lim_{N\to \infty}\Pb\left(  \mathcal{L}\left(\bar{\xi}^{\infty}_{g(\xi^{N},c)}\right)\geq b_{N}-k_{N} -N^{\kappa}\right)
\\&\geq  \lim_{N\to \infty}\Pb\left(  \mathcal{L}\left(\hat{\xi}^{\infty}_{g(\xi^{N},c)}\right)\geq b_{N}-k_{N} -N^{\kappa}\right)
\end{align*}
where $\hat{\xi}^{\infty}$ was defined in Theorem \ref{theorem5} and  for the last inequality we used that under the basic coupling we have $\LL(\hat{\xi}^{\infty}_{g(\xi^{N},c)})\leq \LL(\bar{\xi}^{\infty}_{g(\xi^{N},c)})$.
Applying Theorem \ref{theorem5} yields 
\begin{align}\label{F6}\lim_{N\to\infty}\Pb(\LL(\zeta^{N}_{\g})\geq b_{N}-k_{N}-N^{\kappa})\geq F(c). 
\end{align}
Noting that $ \LL(\zeta^{N}_{\g})\leq X_{k_N}(\g)$ under the basic coupling, by \eqref{holes} it  follows
\begin{align}\label{F5}\lim_{N\to\infty}\Pb(\LL(\zeta^{N}_{\g})\geq b_{N}-k_{N}-N^{\kappa})\leq F(c),
\end{align}
and combining \eqref{F6},\eqref{F5} finishes the proof. 
\end{proof}
We finish by proving the analogue of Corollary \ref{cor} for the rightmost hole. 
\begin{cor}\label{cor2}We have 
\begin{equation}\label{needed}
\lim_{N\to\infty}\Pb(\RR(\zeta^{N}_{\g})\leq b_{N}-k_{N}+N^{\kappa})=F(c). 
\end{equation}
\begin{proof}
The proof follows the same strategy as the proof of Corollary \ref{cor} and particle-hole duality. 

 We first prove the analogue of Theorem \ref{theorem5}: In $\eta^{\infty}$, we replace all $0's$ to the right of $b_N$  by $1's,$ and call the resulting configuration $\xi^{\infty'}$, i.e.
$$ \xi^{\infty'}=\eta^{\infty}*\mathbf{1}_{\Z_{\leq b_{N}}}+\mathbf{1}_{\Z_{>b_{N}}}.$$
Our goal is to show that for $\kappa\in (0,1/3)$ 
\begin{equation}\label{F1}
\lim_{N\to\infty}\Pb(\RR(\xi^{\infty'}_{\g})\leq b_{N}-k_{N}+N^{\kappa})=F(c). 
\end{equation}

Let now $\mu \in \{0,1\}^{\Z}$, we define a configuration $\ww{\mu}$ as follows: First we turn all $1's$ into $0's$ and vice versa,
i.e. we consider 
$$\mu^{a}:=\mathbf{1}_{\Z}-\mu.$$
This produces an ASEP $(\mu^{a}_t, t \geq 0)$ with particles having a drift to the left. We reverse the drift by defining the reflected configuration 
$$\mu^{b}_t (j):=\mu^{a}_{t}(-j), j\in \Z.$$ The configuration $\ww{\mu}$ is then defined by shifting $\mu^{b}$ to the right by $b_N +1$, i.e.
\begin{equation}\label{hat}\ww{\mu}(j):= \mu^{b}(j-b_N-1).\end{equation}
Then, the configurations $\ww{\xi}^{\infty'}, \ww{\eta}^{\infty}$ play exactly the same role as the configurations $\hat{\xi}^{\infty}, \eta^{\infty}$ in the proof of Theorem \ref{theorem5}. 
Note that we have deterministically
$$\RR(\xi^{\infty'}_t)=-\mathcal{L}(\ww{\xi}^{\infty'}_t)+b_N+1,$$
hence \eqref{F1} becomes 
\begin{equation}\label{F2}
\lim_{N\to\infty}\Pb(\mathcal{L}(\ww{\xi}^{\infty'}_{\g})\geq k_{N}-N^{\kappa}+1)=F(c). 
\end{equation}
Let now $\ww{X}_{b_N-k_N}$ be the particle in $\ww{\eta}^{\infty}$ which at time $0$ is located at $ \mathcal{L}(\ww{\xi}^{\infty'}_0)$. 
We label the particles of  $\ww{\eta}^{\infty}$  such that 
$$\cdots\ww{X}_{b_N-k_N+1}<\ww{X}_{b_N-k_N}<\ww{X}_{b_N-k_N-1}\cdots$$
Note that we have , whenever the label $b_N-k_N+j$ exists for $j\in \Z$, 
\begin{equation}\label{label}-(\ww{X}_{b_N-k_N+j}-b_N-1)=H_{b_{N}-k_{N}+j}\end{equation}
with $H_{b_{N}-k_{N}+j}$  being the holes of $\eta^{\infty}$ as in \eqref{particleslab}. 
It thus follows that for $\delta,\kappa \in (0,1/3), j\in \{0, N^{\delta}\}$ we have by 
\eqref{holes}  and \eqref{label}
\begin{equation}\label{F3}
\lim_{N\to\infty}\Pb(\ww{X}_{b_N-k_N+j}(\g)\geq k_{N}-N^{\kappa}+1)=F(c). 
\end{equation}
We may now argue, omitting the details, exactly as in \eqref{forpatrik}  that  $$\lim_{N\to\infty}\Pb\left(\max_{0\leq s\leq \g}\ww{X}_{b_N-k_N+N^{\delta}}(s)-\mathcal{L}(\ww{\xi}^{\infty'}_{s})<0\right)=1.$$ Likewise, we have  the inequality 
$$\ww{X}_{b_N-k_N}(\g)\geq\mathcal{L}(\ww{\xi}^{\infty'}_{\g})$$ under the basic coupling.  Hence \eqref{F2} follows from \eqref{F3}, establishing \eqref{F1}. 

To finish the proof, we need to go from \eqref{F1} to \eqref{needed}, for which it suffices to prove 
\begin{equation}\label{F8}
\lim_{N\to\infty}\Pb(\{\RR(\xi^{\infty'}_{\g})\leq b_{N}-k_{N}+N^{\kappa}\}\setminus\{\RR(\zeta^{N}_{\g})\leq b_{N}-k_{N}+N^{\kappa}\}\})=0.
\end{equation}
Just like in Theorem \ref{theorem6}, we establish the stronger statement
\begin{equation}\label{F9}
\lim_{N\to\infty}\Pb(\{\RR(\xi^{\infty''}_{\g})\leq b_{N}-k_{N}+N^{\kappa}\}\setminus\{\RR(\zeta^{N}_{\g})\leq b_{N}-k_{N}+N^{\kappa}\}\})=0,
\end{equation}
where $\xi^{\infty''}:=\mathbf{1}_{\Z_{<1}}+\mathbf{1}_{\Z_{\geq 1}}*\xi^{\infty'}$.
To establish \eqref{F9},  we may again pass to the configurations $\ww{\xi}^{\infty''},\ww{\zeta}^{N}$, and then establish \eqref{F9} in the same way as Theorem \ref{theorem6}: $\ww{\xi}^{\infty''}$ then plays the role of $\bar{\xi}^{\infty}, $ and $\ww{\zeta}^{N}$ the role of $\zeta^{N}$. 
\end{proof}
 
\end{cor}

\bibliography{Biblio}{}
\bibliographystyle{plain}
\end{document}